\documentclass[12pt,a4paper,reqno]{amsart} %optio titlepage!

\usepackage[T1]{fontenc}       % kirjaimet, joissa aksentteja (skandit)

\usepackage{amsfonts}          % AMS-fontit

\usepackage{amsmath}

\usepackage{amssymb}

\usepackage{overpic}
\usepackage{euscript} % \mathcal tuottaa hyvдnnдkцisen fontin
\usepackage{eurosym}
\usepackage{comment}
\usepackage{mathrsfs} % calligraphic font via the command \mathscr{ABC}
\usepackage{ifthen}
\usepackage{esint} %k.a.integraali komennolla \fint
\usepackage{color}

\long\def\symbolfootnote[#1]#2{\begingroup%
\def\thefootnote{\fnsymbol{footnote}}\footnote[#1]{#2}\endgroup}

{\qed\vspace{5pt}}

\usepackage{amsthm}

\newtheoremstyle{lause}% name
{5pt}% space above
{5pt}% space below
{\slshape}% body font
{\parindent}% indent amount (empty = no indent)
{\bfseries}% theorem head font
{.}% punctuation after theorem head
{.5em}% space after theorem head
{}% theorem head spec (can be left empty, meaning 'normal')

\theoremstyle{lause}
\newtheoremstyle{maaritelma}% name
{5pt}% space above
{5pt}% space below
{\rmfamily}% body font
{\parindent}% indent amount (empty = no indent)
{\bfseries}% theorem head font
{.}% punctuation after theorem head
{.5em}% space after theorem head
{}% theorem head spec (can be left empty, meaning 'normal')

\theoremstyle{maaritelma}
\newtheoremstyle{lause}% name
{5pt}% space above
{5pt}% space below
{\slshape}% body font
{\parindent}% indent amount (empty = no indent)
{\bfseries}% theorem head font
{.}% punctuation after theorem head
{.5em}% space after theorem head
{}% theorem head spec (can be left empty, meaning 'normal')

\theoremstyle{lause}
\newtheorem{theorem}{Theorem}[section]
\newtheorem{lemma}[theorem]{Lemma}

\newtheorem{corollary}[theorem]{Corollary}
\newtheorem{problem}[theorem]{Problem}

\newtheoremstyle{maaritelma}% name
{5pt}% space above
{5pt}% space below
{\rmfamily}% body font
{\parindent}% indent amount (empty = no indent)
{\bfseries}% theorem head font
{.}% punctuation after theorem head
{.5em}% space after theorem head
{}% theorem head spec (can be left empty, meaning 'normal')

\theoremstyle{maaritelma}
\newtheorem{definition}[theorem]{Definition}

\newtheorem{example}[theorem]{Example}
\newtheorem{remark}[theorem]{Remark}

\numberwithin{equation}{section}

\begin{document}

\thispagestyle{empty}

\begin{center}

{\large{\textbf{Inner Riesz pseudo-balayage and its applications to minimum energy problems with external fields}}}

\vspace{18pt}

\textbf{Natalia Zorii}

\vspace{18pt}

\emph{Dedicated to Professor Stephen J.\ Gardiner on the occasion of his 65th birthday}\vspace{8pt}

\footnotesize{\address{Institute of Mathematics, Academy of Sciences
of Ukraine, Tereshchenkivska~3, 01601,
Kyiv-4, Ukraine\\
natalia.zorii@gmail.com }}

\end{center}

\vspace{12pt}

{\footnotesize{\textbf{Abstract.} For the Riesz kernel $\kappa_\alpha(x,y):=|x-y|^{\alpha-n}$ of order $0<\alpha<n$ on $\mathbb
R^n$, $n\geqslant2$, we introduce the so-called inner pseu\-do-bal\-ay\-a\-ge $\hat{\omega}^A$ of a (Radon) measure $\omega$ on $\mathbb
R^n$ to a set $A\subset\mathbb R^n$ as the (unique) measure minimizing the Gauss functional
\[\int\kappa_\alpha(x,y)\,d(\mu\otimes\mu)(x,y)-2\int\kappa_\alpha(x,y)\,d(\omega\otimes\mu)(x,y)\] over
the class $\mathcal E^+(A)$ of
all positive measures $\mu$ of finite energy, concentrated on $A$. For quite general signed $\omega$ (not necessarily of finite
energy) and $A$ (not necessarily closed), such $\hat{\omega}^A$ does exist, and it maintains the basic features of inner
balayage for positive measures (defined when $\alpha\leqslant2$), except for those implied by the domination principle. (To
illustrate the latter, we point out that, in contrast to what occurs for the balayage, the inner pseu\-do-bal\-ay\-a\-ge of a positive measure
may increase its total mass.)
The inner pseu\-do-bal\-ay\-a\-ge $\hat{\omega}^A$ is further shown to be a powerful tool in the problem of minimizing the Gauss functional over
all $\mu\in\mathcal E^+(A)$ with $\mu(\mathbb R^n)=1$, which enables us to improve substantially many recent results on this
topic, by strengthening their formulations and/or by extending the areas of their applications. For instance, if $A$ is a
quasiclosed set of nonzero inner capacity $c_*(A)$, and if $\omega$ is a signed measure,
compactly supported in $\mathbb R^n\setminus{\rm Cl}_{\mathbb R^n}A$, then the problem in question is solvable if and only if
either $c_*(A)<\infty$, or $\hat{\omega}^A(\mathbb R^n)\geqslant1$. In particular, if $c_*(A)=\infty$, then the problem has no
solution  whenever $\omega^+(\mathbb R^n)<1/C_{n,\alpha}$, where $C_{n,\alpha}:=1$ if $\alpha\leqslant2$, and
$C_{n,\alpha}:=2^{n-\alpha}$ otherwise; whereas $\omega^-(\mathbb R^n)$, the total amount of the negative charge, has no influence on this phenomenon. The results obtained are illustrated by some examples.
}}
\symbolfootnote[0]{\quad 2010 Mathematics Subject Classification: Primary 31C15.}
\symbolfootnote[0]{\quad Key words: Minimum Riesz energy problems with external fields, inner Riesz balayage, inner Riesz pseu\-do-bal\-ay\-a\-ge.
}

\vspace{6pt}

\markboth{\emph{Natalia Zorii}} {\emph{Inner Riesz pseudo-balayage and its applications to minimum energy problems with external fields}}

\section{Inner pseudo-balayage: a motivation and a model case}\label{sec-intr}

This paper deals with the theory of potentials with respect to the $\alpha$-Riesz kernels
$\kappa_\alpha(x,y):=|x-y|^{\alpha-n}$ of order $0<\alpha<n$ on $\mathbb R^n$, $n\geqslant2$, $|x-y|$ being the Euclidean
distance in $\mathbb R^n$. Our main goal is to proceed further with the study of minimum $\alpha$-Riesz energy problems in the
presence of
external fields $f$, a point of interest for many researchers (see e.g.\ the monographs \cite{BHS,ST} and references therein,
\cite{Gau,O}, \cite{Z5a}--\cite{ZPot3}, as well as \cite{BDO,Dr0,Z-Oh,Z-Rarx}, some of the most recent papers on this topic).

In the current work we improve substantially many recent results in this field, by strengthening their formulations and/or by
extending the areas of their applications (see Section~\ref{sec-s-uns} for the results obtained). This has become possible
due to the development of a new tool, the inner pseu\-do-bal\-ay\-a\-ge (see Section~\ref{sec-pseudo}, cf.\ also the
present section for a motivation of the proposed definition as well as for a model case).

It is well known that the $\alpha$-Riesz balayage (sweeping out) serves as an efficient tool in the problems in question (see
e.g.\  \cite{Dr0,Z9,ZPot3,Z-Rarx}). However, its application is only limited to the case of $\alpha$ ranging over $(0,2]$, and
to external fields $f$ of the form
\begin{equation}\label{f}
f(x):=-U^\omega(x):=-\int\kappa_\alpha(x,y)\,d\omega(y),
\end{equation}
where $\omega$ is a suitable {\it positive} Radon measure.

To extend the area of application of such a tool to {\it arbitrary} $\alpha\in(0,n)$ and/or to external fields $f$ given by
(\ref{f}), but now with {\it signed} $\omega$ involved, we generalize the standard concept of inner balayage of positive
measures (defined for $\alpha\in(0,2]$) to the so-called inner pseu\-do-bal\-ay\-a\-ge of signed measures, by maintaining the basic
features of the former concept~--- except for those implied by the domination principle.

Being crucial to our study of minimum energy problems with external fields, the concept of inner pseu\-do-bal\-ay\-a\-ge is also of
independent interest,
looking promising for further generalizations and other applications. Before introducing it, we first review some basic facts of
the theory of $\alpha$-Riesz potentials.

We denote by $\mathfrak M$ the linear space of all (real-valued Radon) measures $\mu$ on $\mathbb R^n$, equipped with the
{\it vague} topology of pointwise convergence on the class $C_0(\mathbb R^n)$ of all continuous functions $\varphi:\mathbb
R^n\to\mathbb R$ of compact support, and by $\mathfrak M^+$ the cone of all positive $\mu\in\mathfrak M$, where $\mu$ is {\it positive} if and only if $\mu(\varphi)\geqslant0$ for all positive $\varphi\in C_0(\mathbb R^n)$.
Given $\mu,\nu\in\mathfrak M$, the {\it potential} $U^\mu$ and the {\it mutual energy} $I(\mu,\nu)$ are introduced by
\begin{gather*}
  U^\mu(x):=\int\kappa_\alpha(x,y)\,d\mu(y),\quad x\in\mathbb R^n,\\
  I(\mu,\nu):=\int\kappa_\alpha(x,y)\,d(\mu\otimes\nu)(x,y),
\end{gather*}
respectively, provided that the integral on the right is well defined (as a finite number or $\pm\infty$). For $\mu=\nu$,
$I(\mu,\nu)$ defines the {\it energy} $I(\mu):=I(\mu,\mu)$ of $\mu\in\mathfrak M$.

The following property of {\it strict positive definiteness} for the $\alpha$-Riesz kernels, discovered by M.~Riesz
\cite[Chapter~I, Eq.~(13)]{R} (cf.\ also \cite[Theorem~1.15]{L}), is crucial to the current study: $I(\mu)\geqslant0$ for any (signed)
$\mu\in\mathfrak M$, and $I(\mu)=0\iff\mu=0$. This implies that all (signed) $\mu\in\mathfrak M$ with $I(\mu)<\infty$ form a
pre-Hilbert space $\mathcal E$ with the inner product $\langle\mu,\nu\rangle:=I(\mu,\nu)$ and the energy norm
$\|\mu\|:=\sqrt{I(\mu)}$, see e.g.\ \cite[Lemma~3.1.2]{F1}. The topology on $\mathcal E$ defined by means of this norm, is said
to be {\it strong}.

Another fact decisive to this paper is that the cone $\mathcal E^+:=\mathcal E\cap\mathfrak M^+$ is {\it strongly complete}, and
that the strong topology on $\mathcal E^+$ is {\it finer} than the (induced) vague topology on $\mathcal E^+$ (see J.~Deny
\cite{D1}; for $\alpha=2$, cf.\ also H.~Cartan \cite{Ca1}). Thus any strong Cauchy sequence (net) $(\mu_j)\subset\mathcal E^+$
converges both strongly and vaguely to the same unique limit $\mu_0\in\mathcal E^+$, the strong topology on $\mathcal E$ as well
as the vague topology on $\mathfrak M$ being Hausdorff. (Following B.~Fuglede \cite{F1}, such a kernel is said to be {\it
perfect}.)

\subsection{A model case}\label{sec-model}
As a model case for introducing the concept of inner $\alpha$-Riesz pseu\-do-bal\-ay\-a\-ge, consider first a {\it closed} set $F\subset\mathbb R^n$ and a (signed) measure $\omega\in\mathfrak M$ of {\it finite} energy. Since the class $\mathfrak M^+(F)$ of all $\mu\in\mathfrak M^+$ with the support $S(\mu)\subset F$ is vaguely closed \cite[Section~III.2, Proposition~6]{B2}, the
convex cone $\mathcal E^+(F):=\mathfrak M^+(F)\cap\mathcal E$ is strongly closed, and hence strongly complete, the
$\alpha$-Riesz kernel being perfect. By applying \cite{E2} (Theorem~1.12.3 and Proposition~1.12.4(2)), we therefore conclude that for the
given $\omega\in\mathcal E$, there exists the unique $P\omega\in\mathcal E^+(F)$ such that
\begin{equation}\label{op}
\|\omega-P\omega\|=\min_{\mu\in\mathcal E^+(F)}\,\|\omega-\mu\|,
\end{equation}
and the same $P\omega$ is uniquely characterized within $\mathcal E^+(F)$ by the two relations
\begin{gather}
 \langle P\omega-\omega,\mu\rangle\geqslant0\text{ \ for all $\mu\in\mathcal E^+(F)$},\label{intr-1}\\
\langle P\omega-\omega,P\omega\rangle=0.\label{intr-2}
\end{gather}
This $P\omega$ is said to be the {\it orthogonal projection} of $\omega\in\mathcal E$ onto $\mathcal E^+(F)$.

By a slight modification of \cite[Proof of Theorem~3.1]{Z-bal} we infer from the above that $P\omega$ is the only measure in
$\mathcal E^+(F)$ having the two properties
\begin{align}
 U^{P\omega}&\geqslant U^\omega\text{ \ n.e.\ on $F$},\label{intr-1'}\\
U^{P\omega}&=U^\omega\text{ \ $P\omega$-a.e.,}\label{intr-2'}
\end{align}
where the abbreviation {\it n.e.}\ ({\it nearly everywhere}) means that the inequality holds true everywhere on $F$ except for a
subset $N\subset F$ of {\it inner capacity} zero: $c_*(N)=0$.\footnote{For closed $F$, the set $N$ of all $x\in F$ where
(\ref{intr-1'}) fails is Borel, hence capacitable, and so (\ref{intr-1'}) actually holds true even {\it quasi-everywhere} ({\it
q.e.}) on $F$, namely everywhere on $F$ except for $N$ of {\it outer capacity} zero: $c^*(N)=0$. (For the concepts of inner and
outer capacities, see  \cite[Section~II.2.6]{L}.)}

Assume for a moment that $\alpha\leqslant2$, and that the above $\omega$ is {\it positive}, i.e.\ $\omega\in\mathcal E^+$. By
use of the {\it complete maximum principle} \cite[Theorems~1.27, 1.29]{L}, we derive from (\ref{intr-1'}) and (\ref{intr-2'})
that $P\omega$ is then uniquely characterized within $\mathcal E^+(F)$ by the equality
\begin{equation}\label{PB}
U^{P\omega}=U^\omega\text{ \ n.e.\ on $F$},
\end{equation}
see \cite[Theorem~3.1]{Z-bal}, and hence $P\omega$ is actually the {\it balayage} $\omega^F$ of $\omega\in\mathcal E^+$ onto
$F$:
\[P\omega=\omega^F.\]

$\P$ However, if either $\alpha>2$, or if the above $\omega$ is {\it signed}, then relations (\ref{op})--(\ref{intr-2'}) still
hold, but they no longer result in (\ref{PB}).

Motivated by this observation,  we introduce the following definition.

\begin{definition}
 The {\it pseudo-balayage} $\hat{\omega}^F$ of $\omega\in\mathcal E$ onto a closed set $F\subset\mathbb R^n$ with respect to the
 $\alpha$-Riesz kernel of arbitrary order $\alpha\in(0,n)$ is defined as the only measure in $\mathcal E^+(F)$ satisfying
 (\ref{op}) (with $\hat{\omega}^F$ in place of $P\omega$); or equivalently, as the unique measure in $\mathcal E^+(F)$ having
 properties (\ref{intr-1})--(\ref{intr-2'}) (with $\hat{\omega}^F$ in place of $P\omega$).
\end{definition}

\begin{remark}Assume for a moment that $\omega$ is positive. It follows from the above that the pseu\-do-bal\-ay\-a\-ge $\hat{\omega}^F$
coincides with the balayage $\omega^F$ whenever $\alpha\leqslant2$; while otherwise, the former concept presents a natural extension
of the latter, the problem of balayage for $\alpha>2$ being unsolvable.\footnote{See e.g.\ \cite[Section~IV.5.20]{L}; this is caused by the fact that for $\alpha>2$, the maximum principle fails to hold. Therefore,
when speaking of $\alpha$-Riesz balayage, we understand that $\alpha\in(0,2]$.}
But if now $\omega$ is signed, then for $\alpha\leqslant2$, both $\hat{\omega}^F$ and $\omega^F$ still exist and are unique, whereas, in general,
\[\hat{\omega}^F\ne\omega^F,\] the balayage of signed $\omega$ being defined by linearity (for more details see
Remark~\ref{rem3}).\end{remark}

\begin{remark}\label{rem-gen}In Section~\ref{sec-pseudo} below, we shall extend the above definition of the pseu\-do-bal\-ay\-a\-ge
$\hat{\omega}^F$, given in the model case of $\omega\in\mathcal E$ and closed $F\subset\mathbb R^n$, to:
\begin{itemize}
\item $\omega\in\mathfrak M$ {\it that are not necessarily of finite energy}.
  \item $F\subset\mathbb R^n$ {\it that are not necessarily closed}.
  \end{itemize}
For the former goal, we observe that problem (\ref{op}) is equivalent to that of minimizing {\it the Gauss functional\/}
$\|\mu\|^2-2\int U^\omega\,d\mu$, which makes sense {\it not} only for $\omega\in\mathcal E$.
\end{remark}

We complete this section with some general conventions, used in what follows.

From now on, when speaking of a (signed) measure $\mu\in\mathfrak M$, we understand that its potential $U^\mu$
is well defined and finite almost everywhere with respect to the Lebesgue measure on $\mathbb R^n$; or equivalently (cf.\ \cite[Section~I.3.7]{L}) that
\begin{equation}\label{1.3.10}
\int_{|y|>1}\,\frac{d|\mu|(y)}{|y|^{n-\alpha}}<\infty,
\end{equation}
where $|\mu|:=\mu^++\mu^-$, $\mu^+$ and $\mu^-$ being the positive and negative parts of $\mu$ in the Hahn--Jor\-dan
decomposition. Actually, then (and only then) $U^\mu$ is finite q.e.\ on $\mathbb R^n$, cf.\
\cite[Section~III.1.1]{L}. This would necessarily hold if $\mu$ were required to be {\it bounded} (that is, with $|\mu|(\mathbb R^n)<\infty$), or of finite energy, cf.\ \cite[Corollary to Lemma~3.2.3]{F1}.

A measure $\mu\in\mathfrak M^+$ is said to be {\it concentrated} on a set $A\subset\mathbb R^n$ if $A^c:=\mathbb R^n\setminus A$
is $\mu$-neg\-li\-gible, or equivalently if $A$ is $\mu$-measurable and $\mu=\mu|_A$, $\mu|_A$ being the restriction of $\mu$ to
$A$.
Denoting by $\mathfrak M^+(A)$ the cone of all $\mu\in\mathfrak M^+$ concentrated on $A$, we further write $\mathcal
E^+(A):=\mathfrak M^+(A)\cap\mathcal E$, and let $\mathcal E'(A)$ stand for the closure of $\mathcal E^+(A)$ in the strong
topology on $\mathcal E^+$. We emphasize that $\mathcal E'(A)$ {\it is strongly complete}, being a strongly closed subcone of
the strongly complete cone $\mathcal E^+$.

Given $A\subset\mathbb R^n$, denote by $\mathfrak C_A$ the upward directed set of all compact subsets $K$ of $A$, where
$K_1\leqslant K_2$ if and only if $K_1\subset K_2$. If a net $(x_K)_{K\in\mathfrak C_A}\subset Y$ converges to $x_0\in Y$, $Y$
being a topological space, then we shall indicate this fact by writing
\begin{equation*}\label{abr}x_K\to x_0\text{ \ in $Y$ as $K\uparrow A$}.\end{equation*}

\section{On the inner Riesz balayage}

Before proceeding with an extension of the concept of pseudo-balayage announced in Remark~\ref{rem-gen}, we first recall some
basic facts of the theory of inner $\alpha$-Riesz balayage. Such a theory, generalizing Cartan's pioneering work \cite{Ca2} on
the inner Newtonian balayage ($\alpha=2$) to any $\alpha\in(0,2]$, was initiated in the author's recent papers
\cite{Z-bal,Z-bal2}, and it was further developed in \cite{Z-arx1}--\cite{Z-arx}.\footnote{See also \cite{Z-Deny} for an
application of this theory to Deny's principle of positivity of mass.} Throughout this section, $0<\alpha\leqslant2$.

\begin{definition}[{\rm \cite[Sections~3, 4]{Z-bal}}]\label{def-bal} The {\it inner balayage} $\omega^A$ of a measure
$\omega\in\mathfrak M^+$ to a set $A\subset\mathbb R^n$ is defined as the measure of minimum potential in the class
$\Gamma_{A,\omega}$,
\begin{equation*}\label{gamma}
 \Gamma_{A,\omega}:=\bigl\{\mu\in\mathfrak M^+: \ U^\mu\geqslant U^\omega\text{ \ n.e.\ on $A$}\bigr\}.
\end{equation*}
That is, $\omega^A\in\Gamma_{A,\omega}$ and
\begin{equation}\label{e-d}U^{\omega^A}=\min_{\mu\in\Gamma_{A,\omega}}\,U^\mu\text{ \ on $\mathbb R^n$}.\end{equation}
\end{definition}

\begin{theorem}[{\rm \cite[Sections~3, 4]{Z-bal}}]\label{th-bal}Given arbitrary $\omega\in\mathfrak M^+$ and $A\subset\mathbb
R^n$, the inner balayage $\omega^A$, introduced by Definition~{\rm\ref{def-bal}}, exists and is unique. Furthermore,\footnote{As
pointed out in \cite[Remark~3.12]{Z-bal}, (\ref{ineq1}) no longer characterizes $\omega^A$ uniquely (as it does for closed $A$
and $\omega\in\mathcal E^+$). For more details see footnote~\ref{Fo1}, Corollary~\ref{c-oo'}, Remark~\ref{r-oo'}, and
Theorem~\ref{l-oo'}.}
\begin{align}\label{ineq1}U^{\omega^A}&=U^\omega\text{ \ n.e.\ on\ }A,\\
\notag U^{\omega^A}&\leqslant U^\omega\text{ \ on\ }\mathbb R^n.
\end{align}
The same $\omega^A$ can alternatively be characterized by means of either of the following {\rm(}equivalent{\rm)} assertions:
\begin{itemize}
  \item[{\rm(a)}] $\omega^A$ is the unique measure in $\mathfrak M^+$ satisfying the symmetry relation
  \begin{equation*}\label{eq-sym}
    I(\omega^A,\sigma)=I(\sigma^A,\omega)\text{ \ for all $\sigma\in\mathcal E^+$},
  \end{equation*}
  where $\sigma^A$ denotes the only measure in $\mathcal E'(A)$ with $U^{\sigma^A}=U^\sigma$ n.e.\ on $A$.\footnote{For any
  $\sigma\in\mathcal E^+$ and any $A\subset\mathbb R^n$, the measure $\sigma^A\in\mathcal E'(A)$ having the property
  $U^{\sigma^A}=U^\sigma$ n.e.\ on $A$, exists and is unique. It is, in fact, the orthogonal projection of $\sigma$ in the
  pre-Hil\-bert space $\mathcal E$ onto the convex, strongly complete cone $\mathcal E'(A)$; that is (compare with
  (\ref{op})),
  \[\|\sigma-\sigma^A\|=\min_{\mu\in\mathcal E'(A)}\,\|\sigma-\mu\|.\]
  The same $\sigma^A$ is uniquely characterized within $\mathfrak M^+$ by the extremal property (\ref{e-d}) with
  $\mu:=\sigma$.\label{Fo1}}
  \item[{\rm(b)}] $\omega^A$ is the unique measure in $\mathfrak M^+$ satisfying either of the two limit relations
  \begin{gather*}\omega_j^A\to\omega^A\text{ \ vaguely in $\mathfrak M^+$ as $j\to\infty$},\\
U^{\omega_j^A}\uparrow U^{\omega^A}\text{ \ pointwise on $\mathbb R^n$ as $j\to\infty$},
\end{gather*}
where $(\omega_j)\subset\mathcal E^+$ is an arbitrary sequence having the property\/\footnote{Such $\omega_j\in\mathcal E^+$,
$j\in\mathbb N$, do exist; they can be defined, for instance, by means of the formula
\[U^{\omega_j}:=\min\,\bigl\{U^\omega,\,jU^\lambda\bigr\},\]
$\lambda\in\mathcal E^+$ being fixed (see e.g.\ \cite[p.~272]{L} or \cite[p.~257, footnote]{Ca2}). Here we have used the fact
that for any $\mu_1,\mu_2\in\mathfrak M^+$, there is $\mu_0\in\mathfrak M^+$ such that
$U^{\mu_0}:=\min\,\{U^{\mu_1},\,U^{\mu_2}\}$ \cite[Theorem~1.31]{L}.}
\begin{equation*}\label{eq-mon}U^{\omega_j}\uparrow U^\omega\text{ \ pointwise on $\mathbb R^n$ as
$j\to\infty$},\end{equation*}
whereas $\omega_j^A$ denotes the only measure in $\mathcal E'(A)$ with $U^{\omega_j^A}=U^{\omega_j}$ n.e.\ on\/
$A$.
\end{itemize}
\end{theorem}

\begin{remark}
 For signed $\omega\in\mathfrak M$, we define the inner balayage $\omega^A$ by linearity:
 \begin{equation}\label{lin}
  \omega^A:=(\omega^+)^A-(\omega^-)^A.
 \end{equation}
If moreover the mutual energy $I(\omega,\sigma)$ is well defined for all $\sigma\in\mathcal E$, then this $\omega^A$ is uniquely characterized by the symmetry relation
 \[I(\omega^A,\sigma)=I(\sigma^A,\omega)\text{ \ for all $\sigma\in\mathcal E$,}\]
 which actually only needs to be verified for certain countably many $\sigma\in\mathcal E$, independent of the choice of
 $\omega$ (cf.\ \cite{Z-arx}, Theorem~1.4 and Remark~1.4).
\end{remark}

$\bullet$ In the rest of this paper, we shall always require $A\subset\mathbb R^n$ to have the property\footnote{As shown in
\cite[Theorem~3.9]{Z-Rarx}, $(\mathcal P_1)$ is fulfilled, for instance, if $A$ is {\it quasiclosed} ({\it quasicompact}), that
is, if $A$ can be approximated in outer capacity by closed (compact) sets, see Fuglede \cite{F71}.}
\begin{itemize}
\item[$(\mathcal P_1)$] $\mathcal E^+(A)$ {\it is strongly closed}.
\end{itemize}

Then (and only then)
\[\mathcal E'(A)=\mathcal E^+(A),\]
and hence Theorem~\ref{th-bal} remains valid with $\mathcal E'(A)$ replaced throughout by $\mathcal E^+(A)$. In particular, the
following useful corollary holds true.

\begin{corollary}\label{c-oo'}
 For this $A$ and for any $\omega\in\mathcal E^+$, the inner balayage $\omega^A$ is, in fact, the orthogonal projection of
 $\omega$ onto the {\rm(}convex, strongly complete{\rm)} cone $\mathcal E^+(A)$:
 \[\|\omega-\omega^A\|=\min_{\mu\in\mathcal E^+(A)}\,\|\omega-\mu\|.\]
 The same $\omega^A$ is uniquely characterized
 within $\mathcal E^+(A)$ by $U^{\omega^A}=U^\omega$ n.e.\ on $A$.
\end{corollary}

\begin{remark}\label{r-oo'}
Assumption $(\mathcal P_1)$ is important for the validity of Corollary~\ref{c-oo'}. Indeed, if $\alpha=2$ and $A:=B_r$,
$r\in(0,\infty)$, then for any $\omega\in\mathfrak M^+(B^c_r)$ with $\omega(\overline{B}_r^c)>0$, we have $S(\omega^{B_r})=S_r$ \cite[Theorems~4.1, 5.1]{Z-bal2}, and hence the inner balayage $\omega^{B_r}$ is {\it not} concentrated on the set $B_r$ itself. (Actually, $S(\omega^{B_r})\cap B_r=\varnothing$.)\footnote{Here and in the sequel we use the notations $B_r:=\{|x|<r\}$, $\overline{B}_r:=\{|x|\leqslant r\}$, $S_r:=\{|x|=r\}$.}
\end{remark}

The following generalization of Corollary~\ref{c-oo'} will be useful in the sequel.

\begin{theorem}\label{l-oo'}
Given $\omega\in\mathfrak M^+$, assume that $\omega^A$ is of finite energy.\footnote{This occurs e.g.\ if $\omega$ itself is of
finite energy, or if $\omega$ is bounded and meets the separation condition
\[\inf_{(x,y)\in S(\omega)\times A}\,|x-y|>0.\vspace{-4mm}\]}
Then $\omega^A$ is concentrated on $A$, that is, $\omega^A\in\mathcal E^+(A)$, and it is the unique solution to the problem of
minimizing the Gauss functional
$\|\mu\|^2-2\int U^\omega\,d\mu$, $\mu$ ranging over $\mathcal E^+(A)$. Alternatively, $\omega^A$ is uniquely characterized
within $\mathcal E^+(A)$ by $U^{\omega^A}=U^\omega$ n.e.\ on $A$.
\end{theorem}

\begin{proof} Since $\omega^A=(\omega^A)^A$ (see \cite[Corollary~4.2]{Z-bal}), Corollary~\ref{c-oo'} applied to
$\omega^A\in\mathcal E^+$ shows that, indeed, $\omega^A\in\mathcal E^+(A)$, and hence $\omega^A$ is the orthogonal projection of
itself onto $\mathcal E^+(A)$; or equivalently, it is the (unique) solution to the problem of minimizing the functional
$\|\mu\|^2-2\int U^{\omega^A}\,d\mu$, $\mu$ ranging over $\mathcal E^+(A)$. This implies the former part of the claim by noting that $\int U^{\omega^A}\,d\mu=\int U^\omega\,d\mu$ for all $\mu\in\mathcal E^+(A)$, which, in turn, is derived from (\ref{ineq1}) by use of the fact, to be often used in what follows, that any $\mu$-mea\-s\-ur\-ab\-le subset of $A$ with
$c_*(\cdot)=0$ is $\mu$-negligible for any $\mu\in\mathcal E^+(A)$.

For the latter part, assume (\ref{ineq1}) holds for some $\mu_0\in\mathcal E^+(A)$ in place of $\omega^A$. By the strengthened
version of countable subadditivity for inner capacity (Lemma~\ref{str-sub}),
\[U^{\mu_0}=U^\omega=U^{\omega^A}\text{ n.e.\ on $A$},\]
whence $\mu_0=(\omega^A)^A$, again by Corollary~\ref{c-oo'} applied to $\omega^A\in\mathcal E^+$. Combining this with   $(\omega^A)^A=\omega^A$ (see above) gives $\mu_0=\omega^A$, thereby completing
the whole proof.\end{proof}

\begin{lemma}\label{str-sub}
 For arbitrary $Q\subset\mathbb R^n$ and Borel $U_j\subset\mathbb R^n$,
\[c_*\Bigl(\bigcup_{j\in\mathbb N}\,Q\cap U_j\Bigr)\leqslant\sum_{j\in\mathbb N}\,c_*(Q\cap U_j).\]
\end{lemma}

\begin{proof}
See \cite[pp.~157--158]{F1} (for $\alpha=2$, cf.\ \cite[p.~253]{Ca2}); compare with \cite[p.~144]{L}.
\end{proof}

\section{An extension of the concept of pseudo-balayage}\label{sec-pseudo}

In what follows, we assume that
\begin{equation}\label{cap}
 c_*(A)>0.
\end{equation}
Then (and only then) the class $\mathcal E^+(A)$ is not reduced to $\{0\}$, see \cite[Lemma~2.3.1]{F1}, and the problems in question become nontrivial.

Fix a (signed) measure $\omega\in\mathfrak M$ (not necessarily of finite energy), and define the external field $f:\mathbb
R^n\to[-\infty,\infty]$ by means of the formula
\begin{equation*}
f:=-U^\omega.
\end{equation*}
Being the difference between two lower semicontinuous (l.s.c.)\ functions, $f$ is Borel measurable, and, due to (\ref{1.3.10}),
$f$ is finite q.e.\ on $\mathbb R^n$.

Let $\mathcal E^+_f(A)$ stand for the convex cone of all $\mu\in\mathcal E^+(A)$ such that $f$ is $\mu$-in\-t\-eg\-r\-able; then for
every $\mu\in\mathcal E^+_f(A)$, the Gauss functional\footnote{In constructive function theory, the Gauss functional is also
referred to as {\it the $f$-weighted energy}.}
\begin{equation*}\label{G-f}
I_f(\mu):=\|\mu\|^2+2\int f\,d\mu=\|\mu\|^2-2\int U^\omega\,d\mu
\end{equation*}
is finite. Denoting
\begin{equation}\label{W}
 \hat{w}_f(A):=\inf_{\mu\in\mathcal E^+_f(A)}\,I_f(\mu),
\end{equation}
we have
\begin{equation}\label{West}
 -\infty\leqslant\hat{w}_f(A)\leqslant0,
\end{equation}
the upper estimate being caused by the fact that $0\in\mathcal E^+_f(A)$ while $I_f(0)=0$.

\begin{definition}\label{deff1}A measure $\hat{\omega}^A\in\mathcal E^+_f(A)$ with $I_f(\hat{\omega}^A)=\hat{w}_f(A)$ is said to be {\it the inner $\alpha$-Riesz pseu\-do-bal\-ay\-a\-ge} of $\omega$ onto $A$.\end{definition}

\begin{lemma}\label{l:unique}
 The inner pseudo-balayage $\hat{\omega}^A$ is unique {\rm(}if it exists{\rm)}.
\end{lemma}

\begin{proof}
This follows by standard methods based on the convexity of the class $\mathcal E^+_f(A)$ and the parallelogram identity in the
pre-Hil\-bert space $\mathcal E$, by use of the strict positive definiteness of the Riesz kernel. (See e.g.\ \cite[Proof of
Lemma~6]{Z5a}. Note that this proof requires $\hat{w}_f(A)$ to be finite, which however necessarily holds whenever
$\hat{\omega}^A$ exists.)
\end{proof}

\begin{remark}\label{rem3}If $\omega$ is positive, then, as seen from Theorem~\ref{l-oo'}, the concept of inner pseu\-do-bal\-ay\-a\-ge $\hat{\omega}^A$ extends that of inner balayage $\omega^A$ (introduced for $\alpha\leqslant2$) to arbitrary $\alpha\in(0,n)$.
See the present section as well as Section~\ref{nice} for details.

$\P$ But if $\omega$ is {\it signed}, then the inner balayage $\omega^A$, defined for $\alpha\leqslant2$ by means of (\ref{lin}), may {\it not} coincide with the inner pseu\-do-bal\-ay\-a\-ge $\hat{\omega}^A$. (Thus, in case $\alpha\leqslant2$, the theory of inner pseu\-do-bal\-ay\-a\-ge may be thought of as an alternative theory of inner balayage for signed measures, which is {\it not} equivalent to the standard one.)

Indeed, take $\omega\in\mathfrak M$ such that $\omega=-\omega^-\ne0$. Then $\omega^A=-(\omega^-)^A\ne0$, whereas
$\hat{\omega}^A$, minimizing $\|\mu\|^2+2\int U^{\omega^-}\,d\mu\geqslant0$ over $\mathcal E^+_f(A)$, is obviously $0$, and so $\hat{\omega}^A\ne\omega^A$.
 \end{remark}

\begin{remark}\label{q}It follows easily from Definition~\ref{deff1} that, if $\hat{\omega}^A$ exists, then so does
$(\widehat{c\omega})^A$ for any $c\in[0,\infty)$, and moreover
\begin{equation}\label{equq}
(\widehat{c\omega})^A=c\hat{\omega}^A.
\end{equation}
However, this fails to hold if $c<0$, cf.\ Remark~\ref{rem3}.
\end{remark}

\subsection{On the  existence of the inner pseudo-balayage}\label{sec-cond}
Recall that we are working under the permanent requirements $(\mathcal P_1)$ and (\ref{cap}).
In the rest of this paper, we also assume that $\omega\in\mathfrak M$ satisfies either of
the following properties:\footnote{$(\mathcal P_3)$ is certainly fulfilled if $\omega\in\mathfrak M$ is compactly supported in
$\overline{A}^c$.}
\begin{itemize}
\item[$(\mathcal P_2)$] $\omega\in\mathcal E$.
\item[$(\mathcal P_3)$] {\it $\omega^+$ is bounded, $U^\omega$ is upper semicontinuous on $\overline{A}:={\rm Cl}_{\mathbb
    R^n}A$, and}
     \begin{equation}\label{MA}
      M_A:=\sup_{x\in A}\,U^{|\omega|}(x)<\infty.
     \end{equation}
\end{itemize}

Note that in case $(\mathcal P_2)$, the class $\mathcal E^+_f(A)$ actually coincides with the whole of $\mathcal E^+(A)$, cf.\ (\ref{reprr}),
while in case $(\mathcal P_3)$, it necessarily contains all {\it bounded} $\mu\in\mathcal E^+(A)$.

\begin{theorem}\label{th-ps1} The inner pseudo-balayage $\hat{\omega}^A$, introduced by Definition\/~{\rm\ref{deff1}}, does
exist. Hence,
\begin{equation}\label{p2}
-\infty<\hat{w}_f(A)\leqslant0.
\end{equation}
The same $\hat{\omega}^A$ can alternatively be characterized by either of the following {\rm(a)} or {\rm(b)}:
\begin{itemize}
\item[{\rm(a)}] $\hat{\omega}^A$ is the only measure in $\mathcal E^+_f(A)$ having the two properties
 \begin{gather}\label{def1'}
\int U^{\hat{\omega}^A-\omega}\,d\mu\geqslant0\text{ \ for all $\mu\in\mathcal E^+_f(A)$},\\
\int U^{\hat{\omega}^A-\omega}\,d\hat{\omega}^A=0.\label{def2'}
\end{gather}
\item[{\rm(b)}] $\hat{\omega}^A$ is the only measure in $\mathcal E^+(A)$ having the two properties\/\footnote{Due to
    (\ref{def1}), $U^{\hat{\omega}^A}\geqslant U^\omega$ holds true $\hat{\omega}^A$-a.e.; hence, (\ref{def2}) is equivalent to the
    apparently weaker relation $U^{\hat{\omega}^A}\leqslant U^\omega$ $\hat{\omega}^A$-a.e. Similarly, (\ref{def2'}) can be replaced by $\int U^{\hat{\omega}^A-\omega}\,d\hat{\omega}^A\leqslant0$.}
  \begin{align}\label{def1}
  U^{\hat{\omega}^A}&\geqslant U^\omega\text{ \ n.e.\ on $A$},\\
  U^{\hat{\omega}^A}&=U^\omega\text{ \ $\hat{\omega}^A$-a.e.}\label{def2}
  \end{align}
\end{itemize}
\end{theorem}

\begin{proof} Fixing $\nu\in\mathcal E^+_f(A)$, we shall first show that (\ref{def1}) and (\ref{def2}) hold true for $\nu$ in place of
$\hat{\omega}^A$ if and only if so do (\ref{def1'}) and (\ref{def2'}). It is enough to verify only the "if" part of this claim, the
opposite being obvious from the fact that any $\mu$-mea\-s\-ur\-ab\-le subset of $A$ with
$c_*(\cdot)=0$ is $\mu$-negligible for any $\mu\in\mathcal E^+(A)$.

Assuming, therefore, that (\ref{def1'}) and (\ref{def2'}) hold true (for $\nu$ in place of $\hat{\omega}^A$), suppose to the
contrary that (\ref{def1}) fails. But then there is compact $K\subset A$ such that $U^\nu<U^\omega$ on
$K$ while $c(K)>0$,\footnote{If $A$ is capacitable (e.g.\ Borel), we write $c(A):=c_*(A)=c^*(A)$.} hence
$\int U^{\nu-\omega}\,d\tau<0$ for any $\tau\in\mathcal E^+(K)$, $\tau\ne0$, which contradicts (\ref{def1'}), $\int f\,d\tau$
being obviously finite. Thus (\ref{def1}) does indeed hold, whence
\begin{equation}\label{pr-in}
U^\nu\geqslant U^\omega\text{ \ $\nu$-a.e.}
\end{equation}

Further, assuming to the contrary that (\ref{def2}) fails to hold, we infer from (\ref{pr-in}) that there exists
compact $Q\subset A$ such that $\nu(Q)>0$ while $U^\nu>U^\omega$ on $Q$. Together with (\ref{pr-in}), this yields $\int
U^{\nu-\omega}\,d\nu>0$, which however contradicts (\ref{def2'}).

The equivalence thereby verified enables us to prove the statement on the uniqueness in each of assertions (a) and (b). Indeed, suppose that (\ref{def1}) and (\ref{def2}) are fulfilled by some
$\nu,\nu'\in\mathcal E^+(A)$ in place of $\hat{\omega}^A$. Noting from
(\ref{def2}) that then necessarily $\nu,\nu'\in\mathcal E^+_f(A)$, we conclude by applying (\ref{def1'}) and (\ref{def2'}) to
each of $\nu$ and $\nu'$ that
 \[\langle\nu,\nu'\rangle\geqslant\int U^\omega\,d\nu'=\|\nu'\|^2,\quad\langle\nu',\nu\rangle\geqslant\int
 U^\omega\,d\nu=\|\nu\|^2.\]
Therefore,
  \[0\leqslant\|\nu-\nu'\|^2=\bigl(\|\nu\|^2-\langle\nu',\nu\rangle\bigr)+
  \bigl(\|\nu'\|^2-\langle\nu,\nu'\rangle\bigr)\leqslant0,\]
whence $\nu=\nu'$, by virtue of the strict positive definiteness of the $\alpha$-Riesz kernel.\smallskip

{\it Case $(\mathcal P_2)$.} Assume first that $(\mathcal P_2)$ holds; then the Gauss functional has the form
\begin{equation}\label{repr}
I_f(\mu)=\|\omega-\mu\|^2-\|\omega\|^2\text{ \ for all $\mu\in\mathcal E^+(A)$},
\end{equation}
whence
\begin{equation}\label{reprr}\mathcal E^+_f(A)=\mathcal E^+(A).\end{equation}
Thus the problem on the existence of the inner pseudo-balayage $\hat{\omega}^A$ is reduced to
that on the existence of the orthogonal projection of $\omega\in\mathcal E$ onto $\mathcal E^+(A)$, i.e.
\begin{equation*}\label{pr}
 \hat{\omega}^A\in\mathcal E^+(A)\text{ \ and \ }\|\omega-\hat{\omega}^A\|=\min_{\mu\in\mathcal E^+(A)}\,\|\omega-\mu\|.
\end{equation*}

Since the convex cone $\mathcal E^+(A)$ is strongly closed by $(\mathcal P_1)$, hence strongly complete, $\mathcal E^+$ being
strongly complete by the perfectness of the $\alpha$-Riesz kernel, applying \cite{E2} (Theorem~1.12.3 and Proposition~1.12.4(2)) shows
that such an orthogonal projection does exist, and it is uniquely characterized within $\mathcal E^+(A)$ by both
(\ref{def1'}) and (\ref{def2'}). In view of the equivalence of (a) and (b) proved above, this implies the theorem.\smallskip

{\it Case $(\mathcal P_3)$.} The remaining case $(\mathcal P_3)$ will be treated in four steps.\smallskip

{\it Step 1}. The purpose of this step is to show that the inner pseudo-balayage $\hat{\omega}^A$ exists if and only if there exists the (unique) measure
$\mu_0\in\mathcal E^+_f(A)$ satisfying (a) (equivalently, (b)), and then necessarily $\mu_0=\hat{\omega}^A$.

Assume first that $\hat{\omega}^A$ exists. To verify (\ref{def1}), suppose to the contrary that there is a compact set $K\subset
A$ with $c(K)>0$, such that $U^{\hat{\omega}^A}<U^\omega$ on $K$. A straightforward verification then shows that for any
$\tau\in\mathcal E^+(K)$, $\tau\ne0$, and any $t\in(0,\infty)$,
 \begin{align}\label{step1}
   I_f(\hat{\omega}^A+t\tau)-I_f(\hat{\omega}^A)=2t\int\bigl(U^{\hat{\omega}^A}-U^\omega\bigr)\,d\tau+
   t^2\|\tau\|^2.
 \end{align}
 As $\|\tau\|<\infty$, the value on the right in (\ref{step1}) (hence, also that on the left) is ${}<0$ when $t>0$ is
 small enough, which however contradicts Definition~\ref{deff1}, for $\hat{\omega}^A+t\tau\in\mathcal E^+_f(A)$.

Having thus established (\ref{def1}), we obtain
 \begin{equation}\label{In}
 U^{\hat{\omega}^A}\geqslant U^\omega\text{ \ $\hat{\omega}^A$-a.e.}
 \end{equation}

 Suppose now that (\ref{def2}) fails to hold; then there exists a compact set $Q\subset A$ with $\hat{\omega}^A(Q)>0$, such that
 $U^{\hat{\omega}^A}>U^\omega$ on $Q$, cf.\ (\ref{In}). Denoting $\upsilon:=\hat{\omega}^A|_Q$, we have
 $\hat{\omega}^A-t\upsilon\in\mathcal E^+_f(A)$ for all $t\in(0,1)$, hence
 \begin{align*}
   I_f(\hat{\omega}^A-t\upsilon)-I_f(\hat{\omega}^A)=-2t\int\bigl(U^{\hat{\omega}^A}-
   U^\omega\bigr)\,d\upsilon+t^2\|\upsilon\|^2,
 \end{align*}
 which again contradicts Definition~\ref{deff1} when $t$ is small enough. (Note that, by (\ref{def2}),
 \begin{equation}\label{Sin}
  U^{\hat{\omega}^A}\leqslant U^\omega\text{ \ on $S(\hat{\omega}^A)$},
 \end{equation}
because $U^\omega$ is upper semicontinuous (u.s.c.) on $\overline{A}$ by $(\mathcal P_3)$, while $U^{\hat{\omega}^A}$ is l.s.c.\ on $\mathbb R^n$.)

 For the "if" part of the claim, assume that (a) holds true for some (unique) $\mu_0\in\mathcal E_f^+(A)$. To show that then
 necessarily $\mu_0=\hat{\omega}^A$, we only need to verify that
 \begin{equation}\label{N}
 I_f(\mu)-I_f(\mu_0)\geqslant0\text{ \ for any $\mu\in\mathcal E_f^+(A)$}.
 \end{equation}
 But obviously
  \begin{align*}I_f(\mu)-I_f(\mu_0)&=\|\mu-\mu_0+\mu_0\|^2-2\int U^\omega\,d\mu-\|\mu_0\|^2+2\int U^\omega\,d\mu_0\\
  {}&=\|\mu-\mu_0\|^2+2\int U^{\mu_0-\omega}\,d(\mu-\mu_0),\end{align*}
 and relations (\ref{def1'}) and (\ref{def2'}) (with $\mu_0$ in place of $\hat{\omega}^A$) immediately lead to
 (\ref{N}).\smallskip

 {\it Step 2}. To complete the proof of the theorem, it thus remains to establish the existence of the inner pseu\-do-bal\-ay\-a\-ge
$\hat{\omega}^A$. To this end, suppose throughout this step that $A=K$ is {\it compact}. As $c(K)>0$, see (\ref{cap}), we may restrict ourselves to {\it nonzero} measures $\mu\in\mathcal E^+(K)$. For each of those $\mu$, there are $t\in(0,\infty)$ and $\tau\in\mathcal E^+(K)$ with $\tau(K)=1$ such that $\mu=t\tau$. The potential $U^{|\omega|}$ being bounded on $K$ by $(\mathcal P_3)$,
  \begin{align}\notag
  I_f(\mu)&=t^2\|\tau\|^2-2t\int U^\omega\,d\tau\geqslant t^2\|\tau\|^2-2t\int U^{\omega^+}\,d\tau\\
  {}&\geqslant t^2c(K)^{-1}-2tM_K=t^2\bigl(c(K)^{-1}-2M_Kt^{-1}\bigr),\label{est}
  \end{align}
 $M_K\in[0,\infty)$ being introduced by (\ref{MA}) (with $A:=K$). Hence, by virtue of (\ref{est}), $I_f(\mu)>0$ for all
 $\mu\in\mathcal E^+(K)$ having the property
  \[\mu(K)>2M_Kc(K)=:L_K\in[0,\infty).\]

  On the other hand, $\hat{w}_f(K)\leqslant0$, cf.\ (\ref{West}). In view of the above, $\hat{w}_f(K)$ would therefore be the
  same if $\mathcal E_f^+(K)$ in (\ref{W}) were replaced by
  \begin{equation}\label{LL}
 \mathcal E^+_{L_K}(K):=\bigl\{\mu\in\mathcal E^+(K):\ \mu(K)\leqslant L_K\bigr\},\end{equation}
 that is,
 \begin{equation}\label{L}
 \hat{w}_f(K)=\inf_{\mu\in\mathcal E^+_{L_K}(K)}\,I_f(\mu)=:\hat{w}_{f,L_K}(K).
 \end{equation}
Hence,
\[-\infty<-2M_KL_K\leqslant\hat{w}_f(K)\leqslant0.\]

 Choose a (minimizing) sequence $(\mu_j)\subset\mathcal E^+_{L_K}(K)$ such that
 \[\lim_{j\to\infty}\,I_f(\mu_j)=\hat{w}_{f,L_K}(K).\]
 Being vaguely bounded, cf.\ (\ref{LL}), the sequence $(\mu_j)$ is vaguely relatively compact \cite[Section~III.1,
 Proposition~15]{B2}, and so there is a subsequence $(\mu_{j_k})$ converging vaguely to some $\mu_0\in\mathfrak M^+(K)$. (Here
 we have used the
 first countability of the vague topology on $\mathfrak M$, see \cite[Lemma~4.4]{Z-arx}, as well as the fact that $\mathfrak M^+(K)$ is vaguely closed, see \cite[Section~III.2, Proposition~6]{B2}.) By the principle of descent
 \cite[Eq.~(1.4.5)]{L},
 \[\|\mu_0\|^2\leqslant\liminf_{k\to\infty}\,\|\mu_{j_k}\|^2.\]
 Furthermore, by \cite[Section~IV.4.4, Corollary~3]{B2},
 \[\int U^\omega\,d\mu_0\geqslant\limsup_{k\to\infty}\,\int U^\omega\,d\mu_{j_k}\in(-\infty,\infty),\]
 $U^\omega$ being bounded and u.s.c.\ on the (compact) set $K$. This altogether gives
 \[\hat{w}_{f}(K)\leqslant I_f(\mu_0)\leqslant\liminf_{k\to\infty}\,I_f(\mu_{j_k})=\hat{w}_{f,L_K}(K),\]
 which combined with (\ref{L}) shows that $\mu_0$ serves as the pseudo-balayage $\hat{\omega}^K$.\smallskip

 {\it Step 3}. Our next aim is to show that the constant $L_K$, satisfying (\ref{L}), can be defined to be independent of
 $K\in\mathfrak C_A$ large enough. (To be exact, here and in the sequel $K\in\mathfrak C_A$ is chosen to follow some $K_0$ with
 $c(K_0)>0$.)

By (\ref{def1}) and (\ref{Sin}), $U^{\hat{\omega}^K}=U^\omega$ n.e.\ on $\mathcal S:=S(\hat{\omega}^K)$, hence $\gamma_{\mathcal
S}$-a.e., where $\gamma_{\mathcal S}$ denotes the capacitary measure on $\mathcal S$ (see \cite[Theorem~2.5]{F1}). Since $U^{\gamma_{\mathcal
S}}\geqslant1$ holds true n.e.\ on $\mathcal S$, hence $\hat{\omega}^K$-a.e., we get, by Fubini's
theorem,
 \begin{align*}
   \hat{\omega}^K(\mathbb R^n)&=\int1\,d\hat{\omega}^K\leqslant
   \int U^{\gamma_{\mathcal S}}\,d\hat{\omega}^K=\int U^{\hat{\omega}^K}\,d\gamma_{\mathcal S}\\{}&=\int
   U^\omega\,d\gamma_{\mathcal S}=
   \int U^{\gamma_{\mathcal S}}\,d\omega\leqslant\int U^{\gamma_{\mathcal S}}\,d\omega^+.
    \end{align*}
As $U^{\gamma_{\mathcal S}}\leqslant1$ on $S(\gamma_{\mathcal S})$, applying the Frostman maximum principle if
$\alpha\leqslant2$ (see \cite[Theorem~1.10]{L}), or \cite[Theorem~1.5]{L} otherwise, gives
\begin{equation}\label{La}
\hat{\omega}^K(\mathbb R^n)\leqslant C_{n,\alpha}\omega^+(\mathbb R^n)=:L\text{ \ for all compact $K\subset A$},\end{equation}
where
\begin{equation}\label{Lambda}
C_{n,\alpha}:=\left\{
\begin{array}{cl} 1&\text{if \ $\alpha\leqslant2$},\\
2^{n-\alpha}&\text{otherwise}.\\ \end{array} \right.\end{equation}

We are thus led to the following conclusion, crucial to our proof.
\begin{itemize}
\item[$\P$] {\it We have
\begin{equation}\label{KA}
\hat{w}_f(K)=\hat{w}_{f,L}(K)\in[-2M_AL,0]\text{ \ for all\/ $K\in\mathfrak C_A$,}
\end{equation}
$M_A\in[0,\infty)$ and $L\in[0,\infty)$ being introduced by {\rm(\ref{MA})} and {\rm(\ref{La})}, respectively.
The infimum $\hat{w}_{f,L}(K)$ is an actual minimum with the minimizer $\hat{\omega}^K$.}\smallskip
 \end{itemize}

 {\it Step 4}. To establish the existence of $\hat{\omega}^A$ for noncompact $A$, we first note that the net $\bigl(\hat{w}_{f,L}(K)\bigr)_{K\in\mathfrak C_A}$ decreases, and moreover, by (\ref{KA}),
\begin{equation}\label{lim}
\infty<\lim_{K\uparrow A}\,\hat{w}_{f,L}(K)\leqslant0.
\end{equation}

For any compact $K,K'\subset A$ such that $K\subset K'$ and $c(K)>0$,
\[(\hat{\omega}^K+\hat{\omega}^{K'})/2\in\mathcal E_L^+(K'),\]
whence
\[\|\hat{\omega}^K+\hat{\omega}^{K'}\|^2-4\int
U^\omega\,d(\hat{\omega}^K+\hat{\omega}^{K'})\geqslant4\hat{w}_{f,L}(K')=4I_f(\hat{\omega}^{K'}).\]
Applying the parallelogram identity to $\hat{\omega}^K,\hat{\omega}^{K'}\in\mathcal E^+$ we therefore get
\begin{equation}\label{fund}
 \|\hat{\omega}^K-\hat{\omega}^{K'}\|^2\leqslant2I_f(\hat{\omega}^K)-2I_f(\hat{\omega}^{K'}).
\end{equation}
Noting from (\ref{lim}) that the net $\bigl(I_f(\hat{\omega}^K)\bigr)_{K\geqslant K_0}$ is Cauchy in $\mathbb R$, we infer from
(\ref{fund}) that the net $(\hat{\omega}^K)_{K\geqslant K_0}$ is strong Cauchy in $\mathcal E^+(A)$. The cone $\mathcal E^+(A)$
being strongly closed (hence strongly complete) by $(\mathcal P_1)$, there exists $\zeta\in\mathcal E^+(A)$ such that
\begin{equation}\label{conv}
 \hat{\omega}^K\to\zeta\text{ \ strongly (hence vaguely) in $\mathcal E^+(A)$ as $K\uparrow A$.}
\end{equation}
Moreover, $\zeta\in\mathcal E^+_L(A)$, the mapping $\mu\mapsto\mu(\mathbb R^n)$ being vaguely l.s.c.\ on $\mathfrak
M^+$.\footnote{See \cite[Section~IV.1, Proposition~4]{B2} applied to the (positive, l.s.c.) function $1$ on $\mathbb R^n$. It is
useful to point out that in the case where the set in question is {\it compact}, the mapping $\mu\mapsto\int g\,d\mu$ remains
vaguely l.s.c.\ on $\mathfrak M^+(K)$ for {\it any} l.s.c.\ function $g$ (not necessarily positive). This follows by replacing
$g$ by $g':=g+c\geqslant0$, where $c\in(0,\infty)$, a l.s.c.\ function on a compact set being lower bounded, and then by making
use of the vague continuity of the mapping $\mu\mapsto\mu(K)$ on $\mathfrak M^+(K)$.\label{flsc}}

We claim that this $\zeta$ serves as the inner pseudo-balayage $\hat{\omega}^A$. As shown above (Step~1), this will follow
once we verify (\ref{def1}) and (\ref{def2}) for $\zeta$ in place of $\hat{\omega}^A$.

To verify (\ref{def1}) (for $\zeta$ in place of $\hat{\omega}^A$), it is enough to do this for any given compact $K_*\subset A$.
The strong topology on $\mathcal E^+$ being first-countable, in view of (\ref{conv}) there is a subsequence
$(\hat{\omega}^{K_j})_{j\in\mathbb N}$ of the net $(\hat{\omega}^K)_{K\in\mathfrak C_A}$ such that $K_j\supset K_*$ for all $j$,
and
\begin{equation}\label{J}
\hat{\omega}^{K_j}\to\zeta\text{ \ strongly (hence vaguely) in $\mathcal E^+$ as $j\to\infty$.}
\end{equation}
Passing if necessary to a subsequence and changing the notations, we conclude from (\ref{J}), by use of \cite[p.~166,
Remark]{F1}, that
\begin{equation}\label{JJ}U^\zeta=\lim_{j\to\infty}\,U^{\hat{\omega}^{K_j}}\text{ \ n.e.\ on $\mathbb R^n$}.\end{equation}
Applying now (\ref{def1}) to each $\hat{\omega}^{K_j}$, and then letting $j\to\infty$, on account of the countable subadditivity
of inner capacity on Borel sets \cite[p.~144]{L} we infer from (\ref{JJ}) that (\ref{def1}) (with $\zeta$ in place of
$\hat{\omega}^A$) does indeed hold n.e.\ on $K_*$, whence n.e.\ on $A$.

To establish (\ref{def2}) for $\zeta$ in place of $\hat{\omega}^A$, we note from (\ref{Sin}) applied to $K_j$ that
\begin{equation}\label{Kj}
U^{\hat{\omega}^{K_j}}\leqslant U^\omega\text{ \ on $S(\hat{\omega}^{K_j})$},
\end{equation}
$(K_j)$ being the sequence chosen above. Since $(\hat{\omega}^{K_j})$ converges to $\zeta$ vaguely, see (\ref{J}), for every
$x\in S(\zeta)$ there exist a subsequence $(K_{j_k})$ of $(K_j)$ and points $x_{j_k}\in S(\hat{\omega}^{K_{j_k}})$ such that
$x_{j_k}$, $k\in\mathbb N$, approach $x$ as $k\to\infty$. Thus, by (\ref{Kj}),
\[U^{\hat{\omega}^{K_{j_k}}}(x_{j_k})\leqslant U^\omega(x_{j_k})\text{ \ for all $k\in\mathbb N$}.\]
Letting here $k\to\infty$, in view of the upper semicontinuity of $U^\omega$ on $\overline{A}$ and the lower semicontinuity of
the mapping $(x,\mu)\mapsto U^\mu(x)$ on $\mathbb R^n\times\mathfrak M^+$, where $\mathfrak M^+$ is equipped with the vague
topology \cite[Lemma~2.2.1(b)]{F1}, we obtain (\ref{def2}) for $\zeta$ in place of $\hat{\omega}^A$.

This implies that
\begin{equation}\label{xi'}
\zeta=\hat{\omega}^A,
\end{equation}
thereby completing the proof of the whole theorem.\end{proof}

\begin{remark}
Assume for a moment that $\omega$ is positive, and that $A$ is quasiclosed and Borel. If moreover either $\omega\in\mathcal E^+$, or $U^\omega|_A$ is bounded while $c^*(A)<\infty$,\footnote{A quasiclosed set of finite outer capacity is actually quasicompact \cite[Lemma~3.14]{Fu5}.} then Theorem~\ref{th-ps1} can be deduced from Fuglede's result \cite[Theorem~4.10]{Fu5} on the {\it outer} pseu\-do-balayage with respect to a perfect kernel on a locally compact (Hausdorff) space. The methods developed in the above proof are essentially different from those in \cite{Fu5}, which enabled us to establish the existence of the inner Riesz pseu\-do-balayage $\hat{\omega}^A$ for pretty general $\omega$ and $A$, see $(\mathcal P_1)$ and $(\mathcal P_2)$, or $(\mathcal P_1)$ and $(\mathcal P_3)$. In this regard, it is also worth noting that the above proof seems to admit a generalization to suitable perfect kernels on locally compact spaces, which we plan to pursue in future work.
\end{remark}

\subsection{Further properties of the inner pseudo-balayage}\label{sec-further}
The following theorem justifies the term "inner" pseudo-balayage.

\begin{theorem}\label{conv=ps}$\hat{\omega}^K\to\hat{\omega}^A$ strongly and vaguely in $\mathcal E^+$ as $K\uparrow A$.
\end{theorem}

\begin{proof}In case $(\mathcal P_3)$, this follows by substituting (\ref{xi'}) into (\ref{conv}).

It is thus left to consider case $(\mathcal P_2)$. Then $\omega\in\mathcal E$, and therefore $\hat{\omega}^K$, resp.\ $\hat{\omega}^A$, is the
orthogonal projection of $\omega$ onto the (convex, strongly complete) cone $\mathcal E^+(K)$, resp.\ $\mathcal E^+(A)$. A
slight modification of the proof of (\ref{fund}) shows that
\[\|\hat{\omega}^K-\hat{\omega}^{K'}\|^2\leqslant2I_f(\hat{\omega}^K)-2I_f(\hat{\omega}^{K'})\text{ \ whenever $K\subset
K'$}\quad(K,K'\in\mathfrak C_A).\]
Noting that the net $\bigl(\hat{w}_f(K)\bigr)_{K\in\mathfrak C_A}$ is decreasing and, by (\ref{repr}), bounded:
\[-\|\omega\|^2\leqslant \hat{w}_f(K)\leqslant0\text{ \  for all $K\in\mathfrak C_A$},\] we conclude from the above that the net
$(\hat{\omega}^K)_{K\in\mathfrak C_A}\subset\mathcal E^+(A)$ is strong Cauchy, and hence converges strongly and vaguely to some
(unique) $\mu_0\in\mathcal E^+(A)$. This implies
\[\hat{w}_f(A)\leqslant I_f(\mu_0)=\lim_{K\uparrow A}\,I_f(\hat{\omega}^K)=\lim_{K\uparrow A}\,\hat{w}_f(K),\]
the former equality being derived from the strong convergence of $(\hat{\omega}^K)$ to $\mu_0$ by use of (\ref{repr}). To verify
that this $\mu_0$ actually equals $\hat{\omega}^A$, it thus remains to show that
\begin{equation}\label{lll}\lim_{K\uparrow A}\,\hat{w}_f(K)\leqslant\hat{w}_f(A).\end{equation}
But for every $\mu\in\mathcal E^+(A)$,
\[I_f(\mu)=\lim_{K\uparrow A}\,I_f(\mu|_K)\geqslant\lim_{K\uparrow A}\,\hat{w}_f(K),\]
where the equality follows by applying \cite[Lemma~1.2.2]{F1} to each of the positive, l.s.c., $\mu$-integrable functions
$\kappa_\alpha$, $U^{\omega^+}$, and $U^{\omega^-}$, the set $A$ being $\mu$-measurable. Letting now $\mu$ range over $\mathcal
E^+(A)$ we get (\ref{lll}), thereby completing the proof of the theorem.\end{proof}

\begin{corollary}\label{cor-lambda}If $U^\omega$ is u.s.c.\ on $\overline{A}$ {\rm(}which holds in particular in case $(\mathcal P_3)${\rm)}, then
\begin{equation}\label{tmo}
 \hat{\omega}^A(\mathbb R^n)\leqslant C_{n,\alpha}\omega^+(\mathbb R^n),
\end{equation}
$C_{n,\alpha}$ being introduced by\/ {\rm(\ref{Lambda})}.
\end{corollary}

\begin{proof}
In view of the upper semicontinuity of $U^\omega$ on $\overline{A}$, the proof of (\ref{La}), provided in case $(\mathcal P_3)$, remains
valid in case $(\mathcal P_2)$ as well. Hence, in both cases $(\mathcal P_2)$ and $(\mathcal P_3)$,
\[\hat{\omega}^K(\mathbb R^n)\leqslant C_{n,\alpha}\omega^+(\mathbb R^n)\text{ \ for all $K\in\mathfrak C_A$,}\]
which results in (\ref{tmo}) since the net $(\hat{\omega}^K)_{K\in\mathfrak C_A}$ converges vaguely to $\hat{\omega}^A$
(Theorem~\ref{conv=ps}) while the mapping $\mu\mapsto\mu(\mathbb R^n)$ is vaguely l.s.c.\ on $\mathfrak M^+$ (cf.\
footnote~\ref{flsc}).
\end{proof}

\section{The comparison of the concepts of pseudo-balayage and balayage}\label{nice}

The aim of Examples~\ref{ex1'}--\ref{ex2} below is to demonstrate that usual nice properties of the inner balayage may fail to hold when dealing with the inner pseu\-do-bal\-ay\-a\-ge, and this occurs even in the simplest case of the Dirac measure and a sphere.

$\P$ The first fact illustrating the difference between these two concepts is that
the inner pseu\-do-bal\-ay\-a\-ge may increase the total mass of a positive measure (see e.g.\ (\ref{dr''}), pertaining to $\alpha>2$). Recall that, if $\alpha\in(0,2]$, then, by \cite[Corollary~4.9]{Z-bal},
\begin{equation}\label{tmst}
\mu^A(\mathbb R^n)\leqslant\mu(\mathbb R^n)\text{ \ for any $\mu\in\mathfrak M^+$ and $A\subset\mathbb R^n$}.
\end{equation}

$\P$ The second one is that, for $\alpha>2$, there exist a set $A\subset\mathbb R^n$ which is not inner $\alpha$-thin at infinity\footnote{By \cite[Definition~3.1]{KM}, $A\subset\mathbb R^n$ is said to be {\it inner $\alpha$-thin at infinity} if
\begin{equation*}\label{sum}
 \sum_{j\in\mathbb N}\,\frac{c_*(A_j)}{q^{j(n-\alpha)}}<\infty,
 \end{equation*}
where $q\in(1,\infty)$ and $A_j:=A\cap\{x\in\mathbb R^n:\ q^j\leqslant|x|<q^{j+1}\}$. For $\alpha\in(0,2]$, see also \cite[Definition~2.1]{Z-bal2}, while for $\alpha=2$ and Borel $A$, see \cite[pp.~175--176]{Doob}.} and a measure $\mu\in\mathfrak M^+$ such that
\[\hat{\mu}^A(\mathbb R^n)>\mu(\mathbb R^n),\]
see e.g.\ (\ref{dr''}). In contrast to that, not being inner $\alpha$-thin at infinity is necessary and sufficient for equality to prevail in (\ref{tmst}) for {\it all\/} $\mu\in\mathfrak M^+$, see \cite[Corollary~5.3]{Z-bal2}.

\begin{example}\label{ex1'}Let $D\subset\mathbb R^n$ be a bounded (connected, open) domain, $\omega:=\varepsilon_{x_0}$, where $\varepsilon_{x_0}$ is the unit Dirac measure at $x_0\in D$, and let $A$ be the inverse of $\overline{D}\setminus\{x_0\}$ with respect to the sphere $S_{x_0,1}:=\{|x-x_0|=1\}$.
For these $A$ and $\omega$, $(\mathcal P_1)$ and $(\mathcal P_3)$ are fulfilled, and hence the pseu\-do-bal\-ay\-a\-ge $\hat{\varepsilon}_{x_0}^A$ exists and is unique (Theorem~\ref{th-ps1}).

Assume first that $\alpha\in(0,2]$. Since the balayage $\varepsilon_{x_0}^A$ of $\varepsilon_{x_0}$ onto $A$ is obviously of finite energy, Theorem~\ref{l-oo'} yields
  \begin{equation}\label{u0''}\hat{\varepsilon}_{x_0}^A=\varepsilon_{x_0}^A.\end{equation}
 Applying \cite[Section~IV.5.20]{L} we therefore infer that the pseu\-do-bal\-ay\-a\-ge $\hat{\varepsilon}_{x_0}^A$ is actually the
 Kelvin transform $\gamma_{\overline{D}}^*$ of the capacitary measure $\gamma_{\overline{D}}$ on $\overline{D}$ with
 respect to the sphere $S_{x_0,1}$.
 Hence, in particular,
\begin{equation}\label{2n2}
S(\hat{\varepsilon}_{x_0}^A)=\left\{
\begin{array}{cl}\partial A &\text{if \ $\alpha=2$},\\
A&\text{if \ $\alpha<2$},\\ \end{array} \right.\end{equation}
where $\partial A:=\partial_{\mathbb R^n}A$. The set $A$ not being $\alpha$-thin at infinity, we also get, in consequence of (\ref{u0''}) and \cite[Theorem~3.22]{FZ},
  \begin{equation}\label{u''}
  \hat{\varepsilon}_{x_0}^A(\mathbb R^n)=\varepsilon_{x_0}(\mathbb R^n)=1.
 \end{equation}

Let now $\alpha\in(2,n)$. We aim to show that then, in contrast to (\ref{u''}),\footnote{Compare with (\ref{tmnon}).}
 \begin{equation}\label{dr''}
 1=\varepsilon_{x_0}(\mathbb R^n)<\hat{\varepsilon}_{x_0}^A(\mathbb R^n)\leqslant2^{n-\alpha},
 \end{equation}
 the latter inequality being valid by virtue of (\ref{tmo}).

As is known, the capacitary measure $\gamma_{\overline{D}}$ on $\overline{D}$ is the (unique) measure in $\mathcal E^+(\overline{D})$ of (finite) total mass $c(\overline{D})$, and such that (see \cite[Sections~II.1.3, II.3.13]{L})
 \begin{align}
  U^{\gamma_{\overline{D}}}&\geqslant1\text{ \ n.e.\ on $\overline{D}$},\label{K2''}\\
  U^{\gamma_{\overline{D}}}&=1\text{ \ n.e.\ on $S(\gamma_{\overline{D}})$},\label{K0''}\\
   U^{\gamma_{\overline{D}}}&>1\text{ \ on $D$,}\label{K3''}
 \end{align}
(\ref{K3''}) being caused by the fact that for $\alpha\in(2,n)$, the $\alpha$-Riesz potential of a positive measure is superharmonic on $\mathbb R^n$ \cite[Theorem~1.4]{L}.

For the Kelvin transform $\gamma_{\overline{D}}^*$ of $\gamma_{\overline{D}}$, applying \cite[Eqs.~(4.5.2)--(4.5.4)]{L}
yields
\begin{gather}
U^{\gamma_{\overline{D}}^*}(x)=|x-x_0|^{\alpha-n}U^{\gamma_{\overline{D}}}(x^*),\notag\\
I(\gamma_{\overline{D}}^*)=I(\gamma_{\overline{D}}),\notag\\
\gamma_{\overline{D}}^*(\mathbb R^n)=U^{\gamma_{\overline{D}}}(x_0),\label{K4''}
\end{gather}
$x^*$ being the inverse of $x$ with respect to $S_{x_0,1}$. Combined with (\ref{K2''}) and (\ref{K0''}), this shows that
$\gamma_{\overline{D}}^*$ is a measure of the class $\mathcal E^+(A)$ having the properties
\begin{align*}
   U^{\gamma_{\overline{D}}^*}&\geqslant U^{\varepsilon_{x_0}}\text{ \ n.e.\ on $A$,}\\
   U^{\gamma_{\overline{D}}^*}&=U^{\varepsilon_{x_0}}\text{ \ n.e.\ on $S(\gamma_{\overline{D}}^*)$},
   \end{align*}
   whence (\ref{def1}) and (\ref{def2}) with $\omega:=\varepsilon_{x_0}$ and $\hat{\omega}^A:=\gamma_{\overline{D}}^*$. Thus, by virtue of Theorem~\ref{th-ps1},
   \begin{equation*}\label{eq-ex}
   \hat{\varepsilon}_{x_0}^A=\gamma_{\overline{D}}^*,
   \end{equation*}
   which together with (\ref{K3''}) and (\ref{K4''}) establishes (\ref{dr''}). Also note that
   \begin{equation*}\label{2n2'}
S(\hat{\varepsilon}_{x_0}^A)\subset\partial A\end{equation*}
(compare with (\ref{2n2})).
\end{example}

\begin{example}\label{exa1'}
A slight modification of arguments in Example~\ref{ex1'} shows that if $G\subset\mathbb R^n$ is an open, relatively compact set, then for any $\alpha\in(2,n)$ and any $x_0\in G$,
\begin{equation*}\label{drr''}
 1<\hat{\varepsilon}_{x_0}^{G^c}(\mathbb R^n)\leqslant2^{n-\alpha}.
 \end{equation*}
\end{example}

\begin{example}\label{ex}
 Let $\alpha\in(2,n)$, $A:=B_R^c:=\{|x|\geqslant R\}$, $R\in(0,\infty)$, and let $\omega:=\varepsilon_0$, where $\varepsilon_0$ denotes the unit Dirac
 measure at $x=0$. As follows from Example~\ref{ex1'},
 \[\hat{\varepsilon}_0^{B_R^c}=\gamma_{\overline{B}_r}^*,\]
 where $\gamma_{\overline{B}_r}$ is the capacitary measure on the ball $\overline{B}_r$, $r:=1/R$, and $\gamma_{\overline{B}_r}^*$ is the Kelvin transform of $\gamma_{\overline{B}_r}$ with respect to the unit sphere $S_1$.
 Therefore, by symmetry reasons applied to $\gamma_{\overline{B}_r}$, $\hat{\varepsilon}_0^{B_R^c}$ is uniformly distributed over the sphere $S_R$, and such that
 \begin{equation*}\label{dr}
 1<\hat{\varepsilon}_0^{B_R^c}(\mathbb R^n)\leqslant2^{n-\alpha}.
 \end{equation*}
\end{example}

\begin{example}\label{ex2}Let $\alpha\in(2,n)$, $A:=S_R$, $R\in(0,\infty)$, and let $\omega:=\varepsilon_0$. As shown in Example~\ref{ex},
\[S(\hat{\varepsilon}_0^{B_R^c})=S_R,\] which in view of Definition~\ref{deff1} gives
\[\hat{\varepsilon}_0^{S_R}=\hat{\varepsilon}_0^{B_R^c}.\]
Thus $\hat{\varepsilon}_0^{S_R}$ is uniformly distributed over the sphere $S_R$, and such that
\begin{equation*}\label{dr2}
 1<\hat{\varepsilon}_0^{S_R}(\mathbb R^n)\leqslant2^{n-\alpha}.
 \end{equation*}
 Also note that $\hat{\varepsilon}_0^{S_R}$ is, in fact, the Kelvin transform of the capacitary measure $\gamma_{S_r}$ (${}=\gamma_{\overline{B}_r}$), where $r:=1/R$, with respect to the unit sphere $S_1$.
\end{example}

\section{The inner Gauss variational problem}\label{sec-Gauss}

As before, consider a set $A\subset\mathbb R^n$ such that (\ref{cap}) and $(\mathcal P_1)$ are fulfilled, a (signed) measure
$\omega\in\mathfrak M$ satisfying either $(\mathcal P_2)$ or $(\mathcal P_3)$, and the external field $f$ given by
\begin{equation*}
f:=-U^\omega.
\end{equation*}
According to Theorem~\ref{th-ps1}, the problem of minimizing the Gauss functional $I_f(\mu)$,
\[I_f(\mu):=\|\mu\|^2+2\int f\,d\mu=\|\mu\|^2-2\int U^\omega\,d\mu,\]
over the class $\mathcal E^+_f(A)$ of all $\mu\in\mathcal E^+(A)$ with finite $I_f(\mu)$ is uniquely solvable, and its solution
$\hat{\omega}^A$, called the inner pseu\-do-bal\-ay\-a\-ge of $\omega$ onto $A$, is uniquely characterized within $\mathcal E^+_f(A)$
by both (\ref{def1'}) and (\ref{def2'})~--- or, equivalently, by both (\ref{def1}) and (\ref{def2}).

The rest of the paper is to show that the concept of inner pseudo-balayage serves as a powerful tool in the inner Gauss
variational problem, which reads as follows.

\begin{problem}\label{pr-main} Does there exist $\lambda_{A,f}$ minimizing $I_f(\mu)$ within $\breve{\mathcal E}^+(A)$? Here,
\[\breve{\mathcal E}^+(A):=\bigl\{\mu\in\mathcal E^+(A):\ \mu(\mathbb R^n)=1\}.\]
\end{problem}

Recent results on Problem~\ref{pr-main}, which was originated by C.F.~Gauss \cite{Gau}, are reviewed in the monographs \cite{BHS,ST} (see also numerous references therein); for some of the latest researches on this topic, see \cite{Dr0,Z-Oh,Z-Rarx}.

\begin{remark}\label{compared}If $A=K$ is compact while $f$ is l.s.c.\ on $K$, then the existence of the minimizer $\lambda_{K,f}$ can easily be verified, by use of the fact that the class $\breve{\mathcal E}^+(K)$
is vaguely compact, cf.\ \cite[Section~III.1.9, Corollary~3]{B2}, while the Gauss functional $I_f(\cdot)$ is vaguely l.s.c.\ on $\mathfrak M^+(K)$, the
latter being obvious from the principle of descent and the vague lower semicontinuity of the mapping $\mu\mapsto\int f\,d\mu$ on
$\mathfrak M^+(K)$ (footnote~\ref{flsc}).
However, such a proof, based on the vague topology only, is no longer applicable if either $A$ is noncompact, or $f$ is not l.s.c.

To investigate Problem~\ref{pr-main} in the general case where $A$ is noncompact and/or $f$ is not l.s.c.,  we have recently developed an approach based on the systematic use of both the strong and vague topologies on the pre-Hil\-bert space $\mathcal E$,
which utilized essentially the perfectness of the Riesz kernels, see \cite{Z-Rarx}. However, if $c_*(A)=\infty$, then the analysis
performed in \cite{Z-Rarx} was only limited to $\alpha\leqslant2$ and $\omega\geqslant0$, being mainly based on the theory of inner balayage for positive measures.

Motivated by this observation, we generalize the approach, suggested in \cite{Z-Rarx}, to {\it arbitrary} $\alpha\in(0,n)$ and {\it signed} $\omega$,
by use of the theory of inner pseu\-do-bal\-ay\-a\-ge, developed in Sections~\ref{sec-pseudo}, \ref{nice} above. The results thereby established are formulated in
Section~\ref{sec-s-uns},
and are proved in Sections~\ref{secprr1}--\ref{sec-pr2}.  It is worth emphasizing that those results
improve substantially many recent ones from
\cite{Dr0,Z-Rarx} (see Section~\ref{rem-comp} for some details).
\end{remark}

\subsection{Preliminary results} To begin with, observe that under either of assumptions $(\mathcal P_2)$ or $(\mathcal P_3)$, the Gauss
functional $I_f(\mu)$ is finite for all $\mu\in\breve{\mathcal E}^+(A)$. Thus
\begin{equation}\label{breve}
\breve{\mathcal E}^+(A)\subset\mathcal E^+_f(A),
\end{equation}
and therefore
\begin{equation}\label{w}
\min_{\mu\in\mathcal E^+_f(A)}\,I_f(\mu)=:\hat{w}_f(A)\leqslant w_f(A):=\inf_{\mu\in\breve{\mathcal E}^+(A)}\,I_f(\mu).
\end{equation}

\begin{lemma}\label{wfin} $-\infty<w_f(A)<\infty$.
\end{lemma}

\begin{proof}
According to \cite[Lemma~5]{Z5a}, $w_f(A)<\infty$ is equivalent to the inequality
\[c_*\bigl(\{x\in A:\ |f|(x)<\infty\}\bigr)>0,\]
which indeed holds true by virtue of (\ref{cap}) and the fact that $f$ is finite n.e.\ on $\mathbb R^n$. (Here the strengthened
version of countable subadditivity for inner capacity has been utilized, see Lemma~\ref{str-sub}.) Finally, combining (\ref{w}) and (\ref{p2}) gives $w_f(A)>-\infty$.\end{proof}

The solution $\lambda_{A,f}$ to Problem~\ref{pr-main} is {\it unique} (if it exists), which can be proved by use of the convexity of the class
$\breve{\mathcal E}^+(A)$ and the parallelogram identity in the pre-Hil\-bert space $\mathcal E$. Such $\lambda_{A,f}$ is
said to be {\it the inner $f$-weighted equilibrium measure}.

The following theorem, providing characteristic properties of $\lambda_{A,f}$, can be derived from the author's earlier paper
\cite{Z5a} (see Theorems~1, 2 and Proposition~1 therein).

\begin{theorem}\label{th-ch2}For\/ $\lambda\in\breve{\mathcal E}^+(A)$ to be the\/ {\rm(}unique\/{\rm)} solution\/
$\lambda_{A,f}$ to Problem\/~{\rm\ref{pr-main}}, it is necessary and sufficient that either of the following two inequalities be
fulfilled:
\begin{align}U_f^\lambda&\geqslant\int U_f^\lambda\,d\lambda\text{ \ n.e.\ on\/ $A$},\label{1}\\
U_f^\lambda&\leqslant w_f(A)-\int f\,d\lambda\text{ \ $\lambda$-a.e.\ on\/ $A$,}\label{2}
\end{align}
where\/
\[U_f^\lambda:=U^\lambda+f\]
is said to be the\/ $f$-weighted potential of\/ $\lambda$.
If\/ {\rm(\ref{1})} or\/ {\rm(\ref{2})} holds true, then actually
\begin{equation}\label{cc}
\int U_f^{\lambda}\,d\lambda=w_f(A)-\int f\,d\lambda=:c_{A,f}\in(-\infty,\infty),
\end{equation}
$c_{A,f}$ being referred to as the inner $f$-weighted equilibrium constant.\footnote{Similarly to \cite[p.~27]{ST}, $c_{A,f}$
might also be referred to as {\it the inner modified Robin constant}.}
\end{theorem}

\begin{remark}\label{rem1''} If $f$ is l.s.c.\ on $\overline{A}$ (which occurs e.g.\ in case $(\mathcal P_3)$),
then, by (\ref{2}),
\[U_f^{\lambda_{A,f}}\leqslant c_{A,f}\text{ \ on\/ $S(\lambda_{A,f})$},\]
which combined with (\ref{1}) gives
\[U_f^{\lambda_{A,f}}=c_{A,f}\text{ \ n.e.\ on\/ $A\cap S(\lambda_{A,f})$}.\]
\end{remark}

\section{On the existence of $\lambda_{A,f}$ and its properties}\label{sec-s-uns}

In all that follows, except for Corollary~\ref{th3-cor'}, we assume $A$ and $f$ to satisfy the permanent requirements, reminded at the beginning of Section~\ref{sec-Gauss}.

\subsection{On the solvability of Problem~\ref{pr-main}}\label{sec-s-uns1}
Sufficient and/or necessary conditions for the existence of the (unique) solution $\lambda_{A,f}$ to Problem~\ref{pr-main} are established in Theorems~\ref{th-solv1}, \ref{th-solv2}, \ref{th-unsolv}, \ref{th3} and Corollaries~\ref{cor-qu}, \ref{th3-cor2}, \ref{th3-cor}, \ref{th3-cor'} below.

\begin{theorem}\label{th-solv1} For $\lambda_{A,f}$ to exist, it is sufficient that\/\footnote{See Remark~\ref{th-ext} below for
an extension of Theorem~\ref{th-solv1} and Corollary~\ref{cor-qu}.}
\begin{equation}\label{cap-f}
c_*(A)<\infty.
\end{equation}
\end{theorem}

\begin{corollary}\label{cor-qu}
 $\lambda_{A,f}$ does exist whenever $A$ is quasicompact.
\end{corollary}

\begin{proof}
This is obvious since for quasicompact $A$, both $(\mathcal P_1)$ and (\ref{cap-f}) hold true, by virtue of \cite[Theorem~3.9]{Z-Rarx} and \cite[Definition~2.1]{F71}, respectively.
\end{proof}

\begin{theorem}\label{th-solv2} For $\lambda_{A,f}$ to exist, it is sufficient that
\begin{equation}\label{ps-eq}
\hat{\omega}^A(\mathbb R^n)=1,
\end{equation}
$\hat{\omega}^A$ being the inner pseudo-balayage of $\omega$ onto $A$. Furthermore, then
\begin{equation*}\label{lolo}
\lambda_{A,f}=\hat{\omega}^A,\quad w_f(A)=\hat{w}_f(A),\quad c_{A,f}=0,
\end{equation*}
$c_{A,f}$ being the inner $f$-weighted equilibrium constant.
\end{theorem}

$\bullet$ On account of Theorem~\ref{th-solv1}, in the rest of this section we assume that
\begin{equation}\label{cainf}
c_*(A)=\infty.
\end{equation}

$\bullet$ Unless $(\mathcal P_2)$ holds, assume additionally that
\begin{equation}\label{to}
 \lim_{|x|\to\infty,\ x\in\overline{A}}\,U^{\omega^-}(x)=0.
\end{equation}

\begin{theorem}\label{th-unsolv}Problem\/~{\rm\ref{pr-main}} is unsolvable whenever
\begin{equation}\label{sm}
\hat{\omega}^A(\mathbb R^n)<1.
\end{equation}
\end{theorem}

\begin{corollary}\label{th3-cor2}If\/ $\omega^+=0$, then Problem\/~{\rm\ref{pr-main}} is always unsolvable.\end{corollary}

\begin{proof} Indeed, if $\omega=-\omega^-$, then (\ref{sm}) is fulfilled, for $\hat{\omega}^A=0$ (see Remark~\ref{rem3}).
 \end{proof}

\begin{corollary}\label{th3-cor}
If $U^\omega$ is u.s.c.\ on $\overline{A}$, then Problem\/~{\rm\ref{pr-main}} is unsolvable whenever
\begin{equation*}\label{larger'}
\omega^+(\mathbb R^n)<1/C_{n,\alpha},\end{equation*}
$C_{n,\alpha}$ being introduced by {\rm(\ref{Lambda})}.
\end{corollary}

 \begin{proof} This follows from Theorem~\ref{th-unsolv} by use of (\ref{tmo}).
 \end{proof}

 \begin{theorem}\label{th3} Unless $(\mathcal P_2)$ holds, assume $U^\omega$ is continuous\footnote{When speaking of a
 continuous function, we understand that the values are {\it finite} numbers.} on $\overline{A}$ and
 \begin{equation}\label{too}
 \lim_{|x|\to\infty,\ x\in\overline{A}}\,U^{\omega^\pm}(x)=0.
\end{equation}
Then
\begin{equation*}\label{larger}
\lambda_{A,f}\text{\ exists}\iff\hat{\omega}^A(\mathbb R^n)\geqslant1.
\end{equation*}
If moreover
\begin{equation}\label{Larger}
\hat{\omega}^A(\mathbb R^n)>1,
\end{equation}
then\footnote{Compare with Theorem~\ref{th-solv2} as well as with Remark~\ref{corWW}.}
\begin{equation}\label{c0}
\lambda_{A,f}\ne\hat{\omega}^A,\quad w_f(A)\ne\hat{w}_f(A),\quad c_{A,f}\ne0.
\end{equation}
\end{theorem}

\begin{corollary}\label{th3-cor'}Dropping assumption {\rm(\ref{cainf})} as well as all those imposed on $\omega$, consider signed $\omega\in\mathfrak M$, compactly supported in $\overline{A}^c$. Then $\lambda_{A,f}$ exists if and only if either $c_*(A)<\infty$, or $\hat{\omega}^A(\mathbb R^n)\geqslant1$. In particular, $\lambda_{A,f}$ does not exist if both $c_*(A)=\infty$ and $\omega^+(\mathbb R^n)<1/C_{n,\alpha}$ are fulfilled, $C_{n,\alpha}$ being introduced by {\rm(\ref{Lambda})}.\end{corollary}

\begin{proof}This follows by combining Theorems~\ref{cap-f}, \ref{th3} and Corollary~\ref{th3-cor}.\end{proof}

\begin{remark}
Thus, if $\omega\in\mathfrak M$ is compactly supported in $\overline{A}^c$ while $c_*(A)=\infty$, then, according to Corollary~\ref{th3-cor'}, Problem~\ref{pr-main} is unsolvable whenever $\omega^+(\mathbb R^n)$ is small enough, whereas $\omega^-(\mathbb R^n)$, the total amount of the negative charge, has no influence on this phenomenon. (Note that, when appealing to the electrostatic interpretation of the problem, the fact just observed agrees with our physical intuition.)
\end{remark}

\subsection{On the description of the support $S(\lambda_{A,f})$}\label{sec-s-descr} The following Theorem~\ref{th3'} establishes sufficient conditions for the minimizer $\lambda_{A,f}$ to be of compact support, whereas Examples~\ref{ex3} and \ref{ex4} analyze their sharpness.

\begin{theorem}\label{th3'} Under the hypotheses of Theorem\/~{\rm\ref{th3}}, assume moreover that $A$ is not inner
$\alpha$-thin at infinity. Then $S(\lambda_{A,f})$ is compact whenever {\rm(\ref{Larger})} is fulfilled.
\end{theorem}

Examples~\ref{ex3} and \ref{ex4} provide explicit formulae for $S(\lambda_{A,f})$ for some specific $A$ and $f$. The latter equality in (\ref{SS}) as well as both equalities in (\ref{SSS}) show that Theorem~\ref{th3'} would fail in general if $\hat{\omega}^A(\mathbb R^n)>1$ were replaced by $\hat{\omega}^A(\mathbb R^n)=1$.

\begin{example}\label{ex3}
For $A:=B_R^c:=\{|x|\geqslant R\}$, $R\in(0,\infty)$, and for the unit Dirac measure $\varepsilon_0$ at $x=0$, define
 \[\omega:=\varepsilon_0/q,\text{ \ where \ }q:=\hat{\varepsilon}_0^{B_R^c}(\mathbb R^n).\]
As shown in Example~\ref{ex1'}, $q=1$ if $\alpha\leqslant2$, and $q>1$ otherwise.
By (\ref{equq}),
\[\hat{\omega}^{B_R^c}=\hat{\varepsilon}^{B_R^c}_0/q,\] whence
\begin{equation*}\label{SS'}\hat{\omega}^{B_R^c}(\mathbb R^n)=\hat{\varepsilon}^{B_R^c}_0(\mathbb R^n)/q=1.\end{equation*}
Therefore, by virtue of Theorem~\ref{th-solv2}, the solution $\lambda_{B_R^c,f}$ to Problem~\ref{pr-main} with $f:=-U^\omega$ does exist, and moreover
\[\lambda_{B_R^c,f}=\hat{\omega}^{B_R^c}=\hat{\varepsilon}^{B_R^c}_0/q.\]
In view of what was pointed out in Examples~\ref{ex1'} and \ref{ex}, this gives
\begin{equation}\label{SS}S(\lambda_{B_R^c,f})=\left\{
\begin{array}{ll}S_R&\text{if $\alpha\geqslant2$},\\
B_R^c&\text{otherwise}.\\ \end{array} \right.
\end{equation}
\end{example}

 \begin{example}\label{ex4} Let $\overline{B}_{z,r}$ be the closed ball of radius $r$ centered at the point $z:=(0,\ldots,0,-r)$, and let $A_0$ denote the inverse of $\overline{B}_{z,r}\setminus\{0\}$ with respect to the unit sphere $S_1$ centered at the origin $0$. Then $A_0$ is a closed subset of $\mathbb R^n$, not $\alpha$-thin at infinity, such that $0\not\in A_0$, and whose boundary $\partial A_0$ is unbounded.

 In a manner similar to that in Example~\ref{ex1'}, we see that
 \begin{equation}\label{kel}
 \hat{\varepsilon}_0^{A_0}=\gamma_{\overline{B}_{z,r}}^*,
 \end{equation}
 $\gamma_{\overline{B}_{z,r}}^*$ being the Kelvin transform of $\gamma_{\overline{B}_{z,r}}$, the capacitary measure on $\overline{B}_{z,r}$, with respect to $S_1$. Noting that $\gamma_{\overline{B}_{z,r}}^*(\mathbb R^n)=U^{\gamma_{\overline{B}_{z,r}}}(0)$ (cf.\ (\ref{K4''})), whereas $U^{\gamma_{\overline{B}_{z,r}}}(0)=1$, we obtain
 \begin{equation}\label{tmnon}
\hat{\varepsilon}_0^{A_0}(\mathbb R^n)=1\text{ \ for any $\alpha\in(0,n)$}
 \end{equation}
 (compare with (\ref{dr''})). Applying Theorem~\ref{th-solv2} we therefore infer that the solution $\lambda_{A_0,f}$ to Problem~\ref{pr-main} with $A:=A_0$ and $f:=-U^{\varepsilon_0}$ does exist, and moreover
 \begin{equation}\label{sol}
 \lambda_{A_0,f}=\hat{\varepsilon}_0^{A_0}.
 \end{equation}
 By virtue of the description of $S(\gamma_{\overline{B}_{z,r}})$ \cite[Section~II.3.13]{L}, (\ref{kel}) and (\ref{sol}) yield
 \begin{equation}\label{SSS}S(\lambda_{A_0,f})=\left\{
\begin{array}{cl}\partial A_0&\text{if $\alpha\geqslant2$},\\
A_0&\text{otherwise}.\\ \end{array} \right.
\end{equation}
 \end{example}

 \subsection{Remark}\label{rem-comp} The results thereby obtained improve substantially many recent ones from \cite{Dr0,Z-Rarx}, by strengthening their formulations and/or by extending the areas of their applications. For instance, \cite[Corollary~2.6]{Dr0} only deals with closed sets $A$ that are not thin at infinity, and with external fields $f$ of the form $-U^\omega$, where $\omega:=c\varepsilon_{x_0}$, $c\in(0,\infty)$,
 $\varepsilon_{x_0}$ being the unit Dirac measure at $x_0\not\in A$. However, even for these very particular $A$ and $\omega$, all the assertions in \cite[Corollary~2.6]{Dr0} are in general weaker than the relevant ones, established above. This is caused, in particular, by the fact that those assertions from \cite{Dr0} are given in terms of $\omega(\mathbb R^n)$, whereas ours~--- in terms of $\hat{\omega}^A(\mathbb R^n)$ (see Section~\ref{nice} for the relations between these two values).

 Regarding the advantages of our current approach in comparison with that suggested in \cite{Z-Rarx}, see Remark~\ref{compared} above.

\section{Proof of Theorem~\ref{th-solv2}}\label{secprr1}

Due to condition (\ref{ps-eq}),  the inner pseudo-balayage $\hat{\omega}^A$, minimizing $I_f(\mu)$ over the class
$\mathcal E^+_f(A)$, actually belongs to its proper subclass $\breve{\mathcal E}^+(A)$, see (\ref{breve}). Therefore,
\[w_f(A)\leqslant I_f(\hat{\omega}^A)=\hat{w}_f(A),\]
which combined with (\ref{w}) gives
\[I_f(\hat{\omega}^A)=w_f(A)=\hat{w}_f(A).\]
Thus $\hat{\omega}^A$ serves as the (unique) solution to Problem~\ref{pr-main}, i.e.\ $\hat{\omega}^A=\lambda_{A,f}$.
Substituting this equality into (\ref{def2'}), we get
\[
\int U_f^{\lambda_{A,f}}\,d\lambda_{A,f}=\int\bigl(U^{\hat{\omega}^A}-U^\omega\bigr)\,d\hat{\omega}^A=0,
\]
which according to Theorem~\ref{th-ch2} establishes the remaining relation $c_{A,f}=0$.

\section{Proofs of Theorems~\ref{th-solv1} and \ref{th-unsolv}}\label{sec-pr}

\subsection{Extremal measures}\label{sec-ext} Let $\mathbb M_f(A)$ stand for the (nonempty) set of all nets $(\mu_s)_{s\in
S}\subset\breve{\mathcal E}^+(A)$ having the property
\begin{equation}\label{min}
\lim_{s\in S}\,I_f(\mu_s)=w_f(A);
\end{equation}
those nets $(\mu_s)_{s\in S}$ are said to be {\it minimizing} (in Problem~\ref{pr-main}).
Using the finiteness of $w_f(A)$ (Lemma~\ref{wfin}), the convexity of $\breve{\mathcal E}^+(A)$, and the perfectness of the
$\alpha$-Riesz kernel, one can see with the aid of arguments similar to those in \cite[Lemma~4.2]{Z-Rarx} that there exists the
unique $\xi_{A,f}\in\mathcal E^+$ such that, for every $(\mu_s)_{s\in S}\in\mathbb M_f(A)$,
\begin{equation}\label{ext}
 \mu_s\to\xi_{A,f}\text{ \ strongly and vaguely in $\mathcal E^+$ (as $s$ ranges through $S$)}.
\end{equation}
This $\xi:=\xi_{A,f}$ will be referred to as {\it the extremal measure} (in Problem~\ref{pr-main}).

Due to (\ref{ext}), we have
\begin{equation}\label{tin}\xi_{A,f}\in\mathcal E^+(A),\end{equation}
$\mathcal E^+(A)$ being strongly closed by $(\mathcal P_1)$, and moreover
\begin{equation}\label{tm}
 \xi_{A,f}(\mathbb R^n)\leqslant1,
\end{equation}
the mapping $\mu\mapsto\mu(\mathbb R^n)$ being vaguely l.s.c.\ on $\mathfrak M^+$ \cite[Section~IV.1, Proposition~4]{B2}.

The following simple observation is crucial to the proofs given below.

\begin{lemma}\label{ex-min}
Problem\/~{\rm\ref{pr-main}} is solvable if and only if
\begin{equation}\label{nands}
 \xi_{A,f}(\mathbb R^n)=1\text{ \ and \ }I_f(\xi_{A,f})=w_f(A),
\end{equation}
and in the affirmative case
\begin{equation}\label{nandss}\lambda_{A,f}=\xi_{A,f}.\end{equation}
\end{lemma}

\begin{proof}The "if" part is evident by (\ref{tin}), whereas the opposite is implied by the fact that the trivial net
$(\lambda_{A,f})$ is obviously minimizing, and hence converges strongly to both $\lambda_{A,f}$ and $\xi_{A,f}$. Since the
strong topology on $\mathcal E$ is Hausdorff, (\ref{nandss}) follows.\end{proof}

\subsection{Proof of Theorem~\ref{th-solv1}} Fix a minimizing sequence $(\mu_j)\in\mathbb M_f(A)$; by (\ref{ext}),
\begin{equation*}\label{extmu}
 \mu_j\to\xi\text{ \ strongly and vaguely in $\mathcal E^+$ (as $j\to\infty$)},
\end{equation*}
$\xi:=\xi_{A,f}$ being the (unique) extremal measure in Problem~\ref{pr-main}, whence
\begin{equation}\label{minbound}
\sup_{j\in\mathbb N}\,\|\mu_j\|<\infty.
\end{equation}
To show that this $\xi$ serves as the solution to Problem~\ref{pr-main}, it is enough to verify (\ref{nands}).

Since $\mu_j\to\xi$ vaguely, applying \cite[Section~IV.1, Proposition~4]{B2} gives
\begin{equation}\label{11'}
\xi(\mathbb R^n)\leqslant\liminf_{j\to\infty}\,\mu_j(\mathbb R^n)=1,
\end{equation}
whereas \cite[Section~IV.4.4, Corollary~3]{B2} yields
\begin{equation}\label{upper}\int1_K\,d\xi\geqslant\limsup_{j\to\infty}\,\int1_K\,d\mu_j\text{ \ for every compact
$K\subset\mathbb R^n$},\end{equation}
the indicator function $1_K$ of $K$ being bounded, of compact support, and u.s.c.\ on $\mathbb R^n$. Combining (\ref{11'})
and (\ref{upper}) with
\[\xi(\mathbb R^n)=\lim_{K\uparrow\mathbb R^n}\,\int1_K\,d\xi,\]
we get
\[1\geqslant\xi(\mathbb R^n)\geqslant\limsup_{(j,K)\in\mathbb N\times\mathfrak C}\,\int1_K\,d\mu_j=
1-\liminf_{(j,K)\in\mathbb N\times\mathfrak C}\,\int1_{A\setminus K}\,d\mu_j,\]
$\mathbb N\times\mathfrak C$ being the directed product of the directed sets $\mathbb N$ and
$\mathfrak C:=\mathfrak C_{\mathbb R^n}$ \cite[p.~68]{K}. The former relation in (\ref{nands}) will therefore follow once we
establish the equality
\begin{equation}\label{0}
  \liminf_{(j,K)\in\mathbb N\times\mathfrak C}\,\int1_{A\setminus K}\,d\mu_j=0.
\end{equation}

By \cite[Theorem~2.6]{L} applied to $A\setminus K$, $K\in\mathfrak C$ being arbitrarily chosen, there exists the (unique) inner
capacitary measure $\gamma_{A\setminus K}$, minimizing the energy $\|\mu\|^2$ over the (convex) set $\Gamma_{A\setminus K}$
consisting of all $\mu\in\mathcal E^+$ with
\[U^\mu\geqslant1\text{ \  n.e.\ on $A\setminus K$.}\]
For any $K'\in\mathfrak C$ such that $K\subset K'$, we have $\Gamma_{A\setminus K}\subset\Gamma_{A\setminus K'}$, and
\cite[Lemma~2.2]{L} therefore gives
\begin{equation}\label{str-ca}\|\gamma_{A\setminus K}-\gamma_{A\setminus K'}\|^2\leqslant\|\gamma_{A\setminus
K}\|^2-\|\gamma_{A\setminus K'}\|^2.\end{equation}
Since $\|\gamma_{A\setminus K}\|^2=c_*(A\setminus K)$ \cite[Theorem~2.6]{L}, $\|\gamma_{A\setminus K}\|^2$ decreases as $K$
ranges through $\mathfrak C$, which together with (\ref{str-ca}) implies that the net $(\gamma_{A\setminus K})_{K\in\mathfrak
C}\subset\mathcal E^+$ is Cauchy in the strong topology on $\mathcal E^+$. Noting that $(\gamma_{A\setminus K})_{K\in\mathfrak
C}$ converges vaguely to zero,\footnote{Indeed, for any given $\varphi\in C_0(\mathbb R^n)$, there exists an open, relatively compact
set $G\subset\mathbb R^n$ such that $\varphi(x)=0$ for all $x\not\in\overline{G}$. Hence, $\gamma_{A\setminus
K}(\varphi)=0$ for all $K\in\mathfrak C$ with $K\supset\overline{G}$, and the claim follows.} we get
\begin{equation}\label{str-conv}
  \gamma_{A\setminus K}\to0\text{ \ strongly in $\mathcal E^+$ as $K\uparrow\mathbb R^n$,}
\end{equation}
the $\alpha$-Riesz kernel being perfect.

It follows from the above that
\begin{equation}\label{ii}U^{\gamma_{A\setminus K}}\geqslant1_{A\setminus K}\text{ \ n.e.\ on $A\setminus K$},\end{equation}
and, therefore, $\mu_j$-a.e.\ for all $j\in\mathbb N$. Integrating (\ref{ii}) with respect to $\mu_j$ we obtain, by the Cauchy--Schwarz
(Bunyakovski) inequality,
\[\int1_{A\setminus K}\,d\mu_j\leqslant\int U^{\gamma_{A\setminus K}}\,d\mu_j\leqslant\|\gamma_{A\setminus
K}\|\cdot\|\mu_j\|\text{ \ for all $K\in\mathfrak C$ and $j\in\mathbb N$}.\]
Combined with (\ref{minbound}) and (\ref{str-conv}), this gives (\ref{0}), hence $\xi\in\breve{\mathcal E}^+(A)$, and consequently
\begin{equation}\label{nandsl'}
 w_f(A)\leqslant I_f(\xi).
\end{equation}

To complete the proof of the theorem, it remains to verify the latter relation in (\ref{nands}), which in view of
(\ref{min}) and (\ref{nandsl'}) is reduced to the inequality
\begin{equation}\label{nandsl}
 I_f(\xi)\leqslant\lim_{j\to\infty}\,I_f(\mu_j).\end{equation}
 In case $(\mathcal P_2)$, (\ref{nandsl}) follows at once from the strong convergence of $(\mu_j)$ to $\xi$, by applying
 (\ref{repr}) to each of $\mu_j$ and $\xi$. Otherwise, case $(\mathcal P_3)$ holds, and hence $f=-U^{\omega}$ is l.s.c.\ and
 bounded on $\overline{A}$.
 Thus there is $c\in(0,\infty)$ such that $f':=f+c\geqslant0$ on $\overline{A}$, and \cite[Section~IV.1, Proposition~4]{B2}
 applied to $f'$ gives
 \[\int f\,d\xi+c=\int f'\,d\xi\leqslant\liminf_{j\to\infty}\,\int f'\,d\mu_j=\liminf_{j\to\infty}\,\int f\,d\mu_j+c,\]
 the first and last equalities being valid by virtue of $\xi(\mathbb R^n)=\mu_j(\mathbb R^n)=1$. Therefore,
 \[\int f\,d\xi\leqslant\lim_{j\to\infty}\,\int f\,d\mu_j.\]
 Multiplied by $2$, and then added to
 \[\lim_{j\to\infty}\,\|\mu_j\|^2=\|\xi\|^2,\]
 this results in (\ref{nandsl}), thereby completing the whole proof.

 \begin{remark}\label{th-ext}
   A slight generalization of the above proof shows that Theorem~\ref{th-solv1} and, hence, Corollary~\ref{cor-qu} remain valid
   for an external field $f$ represented as the sum
   \[f:=u+U^\vartheta,\]
   where $\vartheta\in\mathcal E$, while $u:\overline{A}\to(-\infty,\infty]$ is l.s.c., bounded from below, and such that
   \[c_*(\{x\in A:\ u(x)<\infty\})>0.\]
 \end{remark}

\subsection{Proof of Theorem~\ref{th-unsolv}} As $c_*(A)=\infty$ by (\ref{cainf}), there are mutually nonintersecting, compact sets
$K_j\subset A$, $j\in\mathbb N$, such that $|x|\geqslant j$ for all $x\in K_j$ and $c(K_j)\geqslant j$. If
$\lambda_j:=\gamma_{K_j}/c(K_j)\in\breve{\mathcal E}^+(K_j)$ denotes the normalized capacitary measure on $K_j$, then
\begin{gather}\label{31}
\|\lambda_j\|\to0\text{ \ as $j\to\infty$},\\
\label{32}\lambda_j\to0\text{ \ vaguely in $\mathcal E^+$ as $j\to\infty$},
\end{gather}
the latter being implied by the fact that for any compact subset $K$ of $\mathbb R^n$, we have $K\cap S(\lambda_j)=\varnothing$ for all $j$
large enough.
Define
\begin{equation}\label{nnu}
\mu_j:=\hat{\omega}^A+c_j\lambda_j,\text{ \ where \ $c_j:=1-\hat{\omega}^A(\mathbb R^n)$.}
\end{equation}
Noting from (\ref{sm}) that
\begin{equation}\label{est2}
0<c_j\leqslant1\text{ \ for all $j$,}
\end{equation}
we get $\mu_j\in\breve{\mathcal E}^+(A)$ for all $j$, whence
\begin{equation}\label{th3-1}
 w_f(A)\leqslant\liminf_{j\to\infty}\,I_f(\mu_j).
\end{equation}

On the other hand, $I_f(\hat{\omega}^A)=\hat{w}_f(A)$ (see Definition~\ref{deff1} and Theorem~\ref{th-ps1}). By means of a straightforward verification, we derive
from (\ref{31}), (\ref{nnu}), and (\ref{est2}) that
\begin{equation}\label{stim1}\limsup_{j\to\infty}\,I_f(\mu_j)\leqslant\limsup_{j\to\infty}\,\Bigl(I_f(\hat{\omega}^A)+2c_j\int
U^{\omega^-}\,d\lambda_j\Bigr)\leqslant\hat{w}_f(A)+2L_\infty,
\end{equation}
where
\begin{equation*}L_\infty:=\limsup_{j\to\infty}\,\int U^{\omega^-}\,d\lambda_j.
\end{equation*}

Let first $(\mathcal P_2)$ take place. Applying the Cauchy--Schwarz inequality to the measures $\omega^-,\lambda_j\in\mathcal
E^+$, and then letting $j\to\infty$, we infer from (\ref{31}) that
\begin{equation}\label{stim2}
 L_\infty=0.
\end{equation}
 Otherwise, $(\mathcal P_3)$ and hence (\ref{to}) must be fulfilled, which again results in (\ref{stim2}), $\lambda_j$ being the
 unit measure supported by $\overline{A}\cap\{|x|\geqslant j\}$.
 Substituting (\ref{stim2}) into (\ref{stim1}), and then combining the inequality thus obtained with (\ref{th3-1}) and (\ref{w}), we get
 \begin{equation}\label{WW}
 \lim_{j\to\infty}\,I_f(\mu_j)=\hat{w}_f(A)=w_f(A),
 \end{equation}
  which shows that the sequence $(\mu_j)$ is, in fact, minimizing in Problem~\ref{pr-main},
 and hence converges both strongly and vaguely to the extremal measure $\xi_{A,f}$:
 \begin{equation*}\label{convext}
  \mu_j\to\xi_{A,f}\text{ \ strongly and vaguely in $\mathcal E^+$ as $j\to\infty$}.
 \end{equation*}
 On account of (\ref{32})--(\ref{est2}), this  gives
 \[\hat{\omega}^A=\xi_{A,f},\]
 the vague topology on $\mathfrak M$ being Hausdorff. Therefore, by (\ref{sm}),
\[\xi_{A,f}(\mathbb R^n)=\hat{\omega}^A(\mathbb R^n)<1,\]
and an application of Lemma~\ref{ex-min} shows that Problem~\ref{pr-main} is indeed unsolvable.

\begin{remark}\label{corWW}
As shown in (\ref{WW}), under the assumptions of Theorem~\ref{th-unsolv}, equality prevails in (\ref{w}), i.e.
\[\hat{w}_f(A)=w_f(A)\]
(compare with Theorems~\ref{th-solv2} and \ref{th3}).
\end{remark}

\section{Proofs of Theorems~\ref{th3} and \ref{th3'}}\label{sec-pr2}

\subsection{Auxiliary results} According to Corollary~\ref{cor-qu}, for every $K\in\mathfrak C_A$ such that $K\geqslant K_0$,
where $c(K_0)>0$, there is the (unique) solution $\lambda_{K,f}$ to Problem~\ref{pr-main} with $A:=K$, whereas by virtue of
Lemma~\ref{l-cont} below, those $\lambda_{K,f}$ form a minimizing net:
\begin{equation}\label{Lmin}(\lambda_{K,f})_{K\geqslant K_0}\in\mathbb M_f(A).\end{equation}

\begin{lemma}\label{l-cont}$w_f(K)\downarrow w_f(A)$ as $K\uparrow A$.
\end{lemma}

\begin{proof}
For any $\mu\in\breve{\mathcal E}^+(A)$, $\mu(K)\uparrow1$ as $K\uparrow A$. Applying \cite[Lemma~1.2.2]{F1} to each of the
(positive, l.s.c., $\mu$-integrable) functions $\kappa_\alpha$, $U^{\omega^+}$, and $U^{\omega^-}$, we therefore get
\begin{equation*}
I_f(\mu)=\lim_{K\uparrow A}\,I_f(\mu|_K)=\lim_{K\uparrow A}\,I_f(\nu_K)\geqslant\lim_{K\uparrow A}\,w_f(K),
\end{equation*}
where
\[\nu_K:=\mu|_K/\mu(K)\in\breve{\mathcal E}^+(K),\] $K\in\mathfrak C_A$ being large enough. Letting now $\mu$ range over
$\breve{\mathcal E}^+(A)$ gives
\[w_f(A)\geqslant\lim_{K\uparrow A}\,w_f(K),\]
whence the lemma, the opposite being obvious by the monotonicity.
\end{proof}

$\bullet$ Unless case $(\mathcal P_2)$ takes place, assume in what follows that the external field $f=-U^\omega$
is continuous on $\overline{A}$, and that (\ref{too}) is fulfilled.

\begin{lemma}\label{l-extr0}For the extremal measure $\xi_{A,f}$, we have
\begin{equation}\label{extw}
I_f(\xi_{A,f})=w_f(A).
\end{equation}
\end{lemma}

\begin{proof} By virtue of (\ref{ext}) and (\ref{Lmin}),
\begin{equation*}\label{extK}
\lambda_{K,f}\to\xi_{A,f}\text{ \ strongly and vaguely in $\mathcal E^+$ as $K\uparrow A$},
\end{equation*}
hence
\begin{equation}\label{extK1}
\lim_{K\uparrow A}\,\|\lambda_{K,f}\|^2=\|\xi_{A,f}\|^2.\end{equation}
If case $(\mathcal P_2)$ takes place, then the strong convergence of $(\lambda_{K,f})_{K\geqslant K_0}$ to $\xi_{A,f}$ yields,
by applying (\ref{repr}) to each of $\lambda_{K,f}$ and $\xi_{A,f}$,
\begin{equation}\label{extK2}\lim_{K\uparrow A}\,I_f(\lambda_{K,f})=I_f(\xi_{A,f}),\end{equation}
whence, by Lemma~\ref{l-cont},
\[I_f(\xi_{A,f})=\lim_{K\uparrow A}\,w_f(K)=w_f(A).\]

In the remaining case $(\mathcal P_3)$, for any $t>0$ choose $r$ so that
\[|f|<t/2\text{ \ on $A\cap\overline{B}_r^c$,}\]
which is possible in view of (\ref{too}). On account of (\ref{tm}),
\begin{equation}\label{integ1}
 \biggl|\int_{\overline{B}_r^c}\,f\,d(\lambda_{K,f}-\xi_{A,f})\biggr|<t\text{ \ for all $K\geqslant K_0$}.
\end{equation}
The above $r$ can certainly be chosen so that
\begin{equation*}\label{norm}\xi_{A,f}(S_r)=0,\end{equation*}
the measure $\xi_{A,f}$ being bounded. Then, according to \cite[Theorem~0.5$'$]{L},
\[\lambda_{K,f}|_{\overline{B}_r}\to\xi_{A,f}|_{\overline{B}_r}\text{ \ vaguely in $\mathfrak M^+$ as $K\uparrow A$},\]
whence, by the continuity of $f$ on $\overline{A}$,
\[
 \lim_{K\uparrow A}\,\int f\,d\lambda_{K,f}|_{\overline{B}_r}=
 \int f\,d\xi_{A,f}|_{\overline{B}_r},
\]
which combined with (\ref{integ1}), taken for $t>0$ arbitrarily small, results in\footnote{In case $(\mathcal P_2)$, (\ref{integ2}) holds true as well, which is obtained by subtracting (\ref{extK1}) from (\ref{extK2}).
Alternatively, it can be derived from the strong convergence of the net $(\lambda_{K,f})_{K\geqslant K_0}$ to
$\xi_{A,f}$, by applying the Cauchy--Schwarz inequality to the measures $\omega,\lambda_{K,f}-\xi_{A,f}\in\mathcal E$.}
\begin{equation}\label{integ2}
\lim_{K\uparrow A}\,\int f\,d\lambda_{K,f}=\int f\,d\xi_{A,f}.
\end{equation}
Multiplied by $2$, and then added to (\ref{extK1}), this yields (\ref{extK2}), whence (\ref{extw}), again by making use of
Lemma~\ref{l-cont}.
\end{proof}

\begin{corollary}\label{cpr-extr}$\lambda_{A,f}$ exists if and only if $\xi_{A,f}(\mathbb R^n)=1$,
and in the affirmative case $\lambda_{A,f}=\xi_{A,f}$.
\end{corollary}

\begin{proof}
  This follows by combining Lemmas~\ref{ex-min} and \ref{l-extr0}.
\end{proof}

\begin{lemma}\label{l-extr}For the extremal measure $\xi=\xi_{A,f}$, we have
\begin{align}\label{eq-extr1}
 U_f^\xi&\geqslant C_\xi\text{ \ n.e.\ on $A$},\\
 U_f^\xi&\leqslant C_\xi\text{ \ on $S(\xi)$},\label{eq-extr2}
\end{align}
where
\begin{equation}\label{Cxi}
C_\xi:=\int U_f^\xi\,d\xi\in(-\infty,\infty).
\end{equation}
\end{lemma}

\begin{proof}In case $(\mathcal P_3)$, the finiteness of $\int U_f^\xi\,d\xi$ follows from the boundedness of $U^{\omega^\pm}$ on $\overline{A}$, the extremal measure $\xi$ being bounded by (\ref{tm}); while in case $(\mathcal P_2)$, it is obvious.

By Theorem~\ref{th-ch2} applied to each $K\in\mathfrak C_A$ large enough,
\begin{align}U_f^{\lambda_{K,f}}&\geqslant c_{K,f}\text{ \ n.e.\ on $K$},\label{1K}\\
U_f^{\lambda_{K,f}}&\leqslant c_{K,f}\text{ \ on $S(\lambda_{K,f})$},\label{2K}
\end{align}
where
\begin{equation*}\label{CK}
c_{K,f}=\int U_f^{\lambda_{K,f}}\,d\lambda_{K,f}.
\end{equation*}
Furthermore, by combining (\ref{extK1}) with (\ref{integ2}),
\begin{equation}\label{lll0}
\lim_{K\uparrow A}\,c_{K,f}=C_\xi,
\end{equation}
$C_\xi$ being the (finite) constant appearing in (\ref{Cxi}).

Fix $K_*\in\mathfrak C_A$. The strong topology on $\mathcal E^+$ being first-countable, one can choose a subsequence
$(\lambda_{K_j,f})_{j\in\mathbb N}$ of the net $(\lambda_{K,f})_{K\in\mathfrak C_A}$ such that
\begin{equation}\label{J0}
 \lambda_{K_j,f}\to\xi\text{ \ strongly (hence vaguely) in $\mathcal E^+$ as $j\to\infty$.}
\end{equation}
There is certainly no loss of generality in assuming that
\[K_*\subset K_j\text{ \ for all $j$,}\]
for if not, we replace $K_j$ by $K_j':=K_j\cup K_*$; then, by the monotonicity of $\bigl(w_f(K)\bigr)$, the sequence
$(\lambda_{K_j',f})_{j\in\mathbb N}$ remains minimizing, and hence also converges strongly to $\xi$.

Due to the arbitrary choice of $K_*\in\mathfrak C_A$, (\ref{eq-extr1}) will follow once we show that
\begin{equation}\label{JJ'}U^\xi_f\geqslant C_\xi\text{ \ n.e.\ on $K_*$}.\end{equation}
Passing if necessary to a subsequence and changing the notations, we conclude from (\ref{J0}), by virtue of \cite[p.~166,
Remark]{F1}, that
\begin{equation}\label{JJ0}U^\xi=\lim_{j\to\infty}\,U^{\lambda_{K_j,f}}\text{ \ n.e.\ on $\mathbb R^n$}.\end{equation}
Applying now (\ref{1K}) to each $K_j$, and then letting $j\to\infty$, on account of (\ref{lll0}) and (\ref{JJ0}) we arrive at
(\ref{JJ'}). (Here the countable subadditivity of inner capacity on Borel sets has been utilized.)

Since $(\lambda_{K_j,f})$ converges to $\xi$ vaguely, see (\ref{J0}), for every $x\in S(\xi)$ there exist a subsequence
$(K_{j_k})$ of the sequence $(K_j)$ and points $x_{j_k}\in S(\lambda_{K_{j_k},f})$, $k\in\mathbb N$, such that $x_{j_k}$ approach $x$ as
$k\to\infty$. Thus, according to (\ref{2K}),
\[U_f^{\lambda_{K_{j_k},f}}(x_{j_k})\leqslant\int U^{\lambda_{K_{j_k},f}}_f\,d\lambda_{K_{j_k},f}\text{ \ for all $k\in\mathbb
N$}.\]
Letting here $k\to\infty$, and applying (\ref{lll0}), the continuity of $f$ on $\overline{A}$, and the lower semicontinuity of the mapping $(x,\mu)\mapsto U^\mu(x)$ on
$\mathbb R^n\times\mathfrak M^+$, $\mathfrak M^+$ being equipped with the vague topology \cite[Lemma~2.2.1(b)]{F1}, we get the remaining inequality (\ref{eq-extr2}).
\end{proof}

\subsection{Proof of Theorem~\ref{th3}} We first remark from Theorems~\ref{th-solv2} and \ref{th-unsolv} that it is actually enough to consider the case when (\ref{Larger}) is fulfilled.

For the extremal measure $\xi=\xi_{A,f}$, we infer from (\ref{eq-extr1}) and (\ref{eq-extr2}) that
\[U_f^\xi=C_\xi\text{ \ n.e.\ on $S(\xi)\cap A$,}\]
whence
\begin{equation}\label{Cxieq}
U_f^\xi=C_\xi\text{ \ $\xi$-a.e.,}
\end{equation}
the measure $\xi$ being of the class $\mathcal E^+(A)$.
We claim that then necessarily
\begin{equation}\label{ne0}
 C_\xi\ne0.
\end{equation}
Indeed, if this were not true, then (\ref{eq-extr1}) and (\ref{Cxieq}) would imply, by virtue of Theorem~\ref{th-ps1}, that
$\xi=\hat{\omega}^A$, which however contradicts (\ref{Larger}), for $\xi(\mathbb R^n)\leqslant1$ by (\ref{tm}).

Integrating (\ref{Cxieq}) with respect to $\xi$ we obtain
\[\int U_f^\xi\,d\xi=C_\xi\cdot\xi(\mathbb R^n),\]
whence, by (\ref{Cxi}) and (\ref{ne0}),
\[\xi(\mathbb R^n)=1.\]
Applying Corollary~\ref{cpr-extr} we see that under assumption (\ref{Larger}), $\lambda_{A,f}$ does indeed exist, and moreover $\lambda_{A,f}=\xi$.

The equality $\lambda_{A,f}=\xi$ implies, by use of (\ref{cc}), (\ref{Cxi}), and (\ref{ne0}), that
\begin{equation*}\label{cne0'}c_{A,f}=\int U_f^{\lambda_{A,f}}\,d\lambda_{A,f}=\int U_f^\xi\,d\xi=C_\xi\ne0,\end{equation*}
which proves the third relation in (\ref{c0}). The first is obvious since, by (\ref{Larger}),
\[\lambda_{A,f}(\mathbb R^n)=1<\hat{\omega}^A(\mathbb R^n).\]
Finally, the first relation implies the second, for if not, then $w_f(A)=\hat{w}_f(A)$, whence $\lambda_{A,f}=\hat{\omega}^A$, by the uniqueness of $\hat{\omega}^A$ and the inclusion $\breve{\mathcal E}^+(A)\subset\mathcal E^+_f(A)$, cf.\ (\ref{breve}).

\subsection{Proof of Theorem~\ref{th3'}} According to Theorem~\ref{th3}, under the stated assumptions there exists
the solution $\lambda_{A,f}$ to Problem~\ref{pr-main}, and moreover, by (\ref{c0}),
\begin{equation}\label{cne0}
c_{A,f}\ne0.
\end{equation}
Assume to the contrary that $S(\lambda_{A,f})$ is noncompact. As seen from Remark~\ref{rem1''}, then there exists a sequence
$(x_j)\subset A$ such that $|x_j|\to\infty$ as $j\to\infty$, and
\[U_f^{\lambda_{A,f}}(x_j)=c_{A,f}\text{ \ for all $j\in\mathbb N$}.\]
On account of (\ref{too}), this yields
\begin{equation*}
 \liminf_{j\to\infty}\,U^{\lambda_{A,f}}(x_j)=c_{A,f},
\end{equation*}
whence $c_{A,f}\geqslant0$, which in view of (\ref{cne0}) shows that, actually,
\begin{equation}\label{li}
c_{A,f}>0.
\end{equation}
But, by (\ref{1}), $U_f^{\lambda_{A,f}}\geqslant c_{A,f}$ n.e.\ on $A$, which together with (\ref{too}) and (\ref{li}) gives
\[\liminf_{|x|\to\infty, \ x\in A}\,U^{\lambda_{A,f}}(x)>0.\]
However, this is impossible in consequence of \cite[Remark~4.12(i)]{KM}, the set $A$ not being inner $\alpha$-thin at infinity.

%\section{A data availability statement} This manuscript has no associated data.


\begin{thebibliography}{99}
\setlength{\parskip}{1.2ex plus 0.5ex minus 0.2ex}

\bibitem{BDO}
Benko, D., Dragnev, P.D., Orive, R.: On point-mass Riesz external fields on the real axis.
J. Math. Anal. Appl. {\bf 491}, 124299 (2020)

\bibitem{BHS} Borodachov, S.V., Hardin, D.P., Saff, E.B.: Discrete Energy on Rectifiable Sets. Springer, Berlin (2019)

\bibitem{B2} Bourbaki, N.: Integration. Chapters~1--6.
Springer, Berlin (2004)

\bibitem{Ca1} Cartan, H.: Th\'eorie du potentiel newtonien: \'energie, capacit\'e, suites de potentiels. Bull. Soc. Math. France
    {\bf 73}, 74--106 (1945)

\bibitem{Ca2} Cartan, H.: Th\'eorie g\'en\'erale du balayage en potentiel newtonien. Ann. Univ. Fourier Grenoble {\bf 22},
    221--280 (1946)

\bibitem{D1} Deny, J.: Les potentiels d'\'energie finie. Acta Math.\ {\bf 82}, 107--183 (1950)

\bibitem{Doob} Doob, J.L.: Classical Potential Theory and Its Probabilistic Counterpart. Springer, Berlin (1984)

\bibitem{Dr0} Dragnev, P.D., Orive, R., Saff, E.B., Wielonsky F.: Riesz energy problems with external fields and related theory.
Constr. Approx. (2022). https://doi.org/10.1007/s00365-022-09588-z

\bibitem{E2}
Edwards, R.E.: Functional Analysis. Theory and Applications. Holt,
Rinehart and Winston, New York (1965)

\bibitem{F1} Fuglede, B.: On the theory of potentials in locally compact spaces. Acta Math. {\bf 103}, 139--215  (1960)

\bibitem{F71} Fuglede, B.: The quasi topology associated with a countably subadditive set function. Ann. Inst. Fourier Grenoble
    {\bf 21}, 123--169 (1971)

\bibitem{Fu5} Fuglede, B.: Symmetric function kernels and sweeping of measures. Anal. Math.
{\bf 42}, 225--259 (2016)

\bibitem{FZ} Fuglede, B., Zorii, N.: Green kernels associated with Riesz kernels. Ann. Acad. Sci. Fenn. Math. {\bf 43}, 121--145
    (2018)

\bibitem{Gau}
Gauss, C.F.: Allgemeine Lehrs\"atze in Beziehung auf die im verkehrten Verh\"altnisse des
  Quadrats der Entfernung wirkenden Anziehungs-- und Absto\ss
  ungs--Kr\"afte (1839). Werke 5, 197--244 (1867)

\bibitem{K}
Kelley, J.L.: General Topology. Princeton, New York (1957)

\bibitem{KM} Kurokawa, T., Mizuta, Y.: On the order at infinity of Riesz potentials. Hiroshima Math. J. {\bf 9}, 533--545 (1979)

\bibitem{L} Landkof, N.S.: Foundations of Modern Potential Theory. Springer, Berlin (1972)

\bibitem{O} Ohtsuka, M.: On potentials in locally compact spaces. J. Sci. Hiroshima Univ. Ser.\ A-I {\bf 25}, 135--352 (1961)

\bibitem{R} Riesz, M.: Int\'egrales de Riemann--Liouville et potentiels. Acta Szeged {\bf 9}, 1--42 (1938)

\bibitem{ST} Saff, E.B., Totik, V: Logarithmic Potentials with External Fields. Springer, Berlin (1997)

\bibitem{Z5a} Zorii, N.V.: Equilibrium potentials with external fields. Ukrainian Math. J. {\bf 55},
1423--1444 (2003)

\bibitem{Z5aa} Zorii, N.V.: Equilibrium problems for potentials with external fields. Ukrainian Math. J. {\bf 55},
1588--1618 (2003)

\bibitem{Z9} Zorii, N.:  Constrained energy problems with external fields for vector measures.  Math. Nachr. {\bf 285},
    1144--1165 (2012)

\bibitem{ZPot2} Zorii, N.: Equilibrium problems for infinite dimensional vector potentials with
external fields. Potential Anal. {\bf 38}, 397--432 (2013)

\bibitem{ZPot3} Zorii, N.:  Necessary and sufficient conditions for the solvability of the Gauss
variational problem for infinite dimensional vector measures. Potential Anal.
{\bf 41}, 81--115 (2014)

\bibitem{Z-bal} Zorii, N.: A theory of inner Riesz balayage and its applications. Bull. Pol. Acad. Sci. Math. {\bf 68}, 41--67
    (2020)

\bibitem{Z-bal2} Zorii, N.: Harmonic measure, equilibrium measure, and thinness at infinity in the theory of Riesz potentials.
    Potential Anal. {\bf 57}, 447--472 (2022)

\bibitem{Z-arx1} Zorii, N.: Balayage of measures on a locally compact space.
Anal. Math. {\bf 48}, 249--277 (2022)

\bibitem{Z-arx-22} Zorii, N.: On the theory of capacities on locally compact spaces and its interaction with the theory of
    balayage. Potential Anal. (2022). https://doi.org/10.1007/s11118-022-10010-3

\bibitem{Z-arx} Zorii, N.: On the theory of balayage on locally compact spaces. Potential Anal. (2022).
    https://doi.org/10.1007/s11118-022-10024-x

\bibitem{Z-Deny} Zorii, N.: On the role of the point at infinity in Deny's principle of positivity of mass for Riesz potentials.
    arXiv:2202.12418 (2022)

\bibitem{Z-Oh} Zorii, N.: Minimum energy problems with external fields on locally compact spaces. arXiv:2207.14342 (2022)

\bibitem{Z-Rarx} Zorii, N.: Minimum Riesz energy problems with external fields. arXiv:2209.05891 (2022)
\end{thebibliography}
\end{document}